\documentclass[12pt]{article}

\newcommand{\RNum}[1]{\uppercase\expandafter{\romannumeral #1\relax}}

\usepackage{xcolor}
\usepackage{amsmath}
\usepackage{amssymb}
\usepackage{latexsym}
\usepackage{epsf}
\usepackage{supertabular}
\usepackage{rotating}
\usepackage{multicol}
\usepackage{multirow}
\usepackage{epsfig}
\usepackage{fancyhdr}       

\newif\ifger
\gerfalse

\newtheorem{theorem}{Theorem}[section]

\newtheorem{lemma}[theorem]{Lemma}

\newtheorem{corollary}[theorem]{Corollary}

\newtheorem{remark}[theorem]{Remark}

\newtheorem{proposition}[theorem]{Proposition}

\usepackage{cases}

\title{Analysing 2-$(v,k,2)$ designs admitting a flag-transitive almost simple automorphism group with socle $PSL(2,q)$ by means of conics and hyperovals of $PG(2,q)$
}
\date{}
\author{Alessandro Montinaro$^1$,Yanwei  Zhao$^2$\footnote{Corresponding author.}, Zhilin Zhang$^3$, Shenglin Zhou$^4$\footnote{This work is
supported by the National Natural Science Foundation of China (Grant No.12271173). Email: alessandro.montinaro@unisalento.it(A. Montinaro), ywzhao@sdau.edu.cn(Y. Zhao), 20241032@gdufe.edu.cn(Z. Zhang), slzhou@scut.edu.cn(S. Zhou)}\\
\small  \it ${}^1$  Dipartimento di Matematica e Fisica “E. De Giorgi”, University of Salento,\\
\small  \it Lecce 73100, Italy\\
\small  \it ${}^2$School of Information Science and Engineering, Shandong Agricultural University,\\
 \small \it Shandong 271018, China\\
\small  \it ${}^3$School of Statistics and Mathematics, Guangdong University of Finance and Economics,\\
\small \it Guangzhou 510320, China\\
\small  \it ${}^4$School of Mathematics, South China University of Technology,\\
\small \it  Guangzhou 510641,  China
}

\begin{document}
\baselineskip=19pt
\maketitle
\date

\begin{abstract}
The classification of the $2$-designs with $\lambda=2$ admitting a flag-transitive automorphism groups with socle $PSL(2,q)$ is completed by settling the two open cases in \cite{ABDT}. The result is achieved by using conics and hyperovals of $PG(2,q)$.

\medskip
\noindent{\bf Mathematics Subject Classification (2020):} 05B05, 05B25, 20B25, 51E15, 51E21

\medskip
\noindent{\bf Keywords:} $2$-design, automorphism group, flag-transitive, socle, projective plane, conic, hyperoval.
\end{abstract}

\section{Introduction}

A $2$-$(v,k,\lambda )$ \emph{design} $\mathcal{D}$ is a pair $(\mathcal{P},%
\mathcal{B})$ with a set $\mathcal{P}$ of $v$ points and a set $\mathcal{B}$
of $b$ blocks such that each block is a $k$-subset of $\mathcal{P}$ and each two
distinct points are contained in $\lambda $ blocks. We say $\mathcal{D}$ is 
\emph{non-trivial} if $2<k<v-1$, and symmetric if $v=b$. All $2$-$(v,k,\lambda
)$ designs in this paper are assumed to be non-trivial. An automorphism of $%
\mathcal{D}$ is a permutation of the point set which preserves the block
set. The set of all automorphisms of $\mathcal{D}$ with the composition of
permutations forms a group, denoted by $\mathrm{Aut(\mathcal{D})}$. For a
subgroup $G$ of $\mathrm{Aut(\mathcal{D})}$, $G$ is said to be \emph{%
point-primitive} if $G$ acts primitively on $\mathcal{P}$, and said to be 
\emph{point-imprimitive} otherwise. In this setting, we also say that $%
\mathcal{D}$ is either \emph{point-primitive} or \emph{point-imprimitive}, respectively. A \emph{flag} of $\mathcal{D}$ is a pair $(x,B)$ where $x$ is
a point and $B$ is a block containing $x$. If $G\leq \mathrm{Aut(\mathcal{D})%
}$ acts transitively on the set of flags of $\mathcal{D}$, then we say that $%
G$ is \emph{flag-transitive} and that $\mathcal{D}$ is a \emph{%
flag-transitive design}.

The $2$-$(v,k,\lambda )$ designs $\mathcal{D}$ admitting a flag-transitive
automorphism group $G$ have been widely studied by several authors. If $\lambda=1$, that is when $\mathcal{D}$ is a linear space, then $G$ acts point-primitively on $\mathcal{D}$ by an important results due to Higman and McLaughlin \cite{HM} dating back to 1961. In 1990, Buekenhout, Delandtsheer, Doyen, Kleidman, Liebeck and Saxl \cite{BDDKLS} obtained a classification of $2$-designs with $\lambda =1$ except when $v$ is a power of a prime and $G \leq A \Gamma L_{1}(v)$. If $\lambda>1$, it is no longer true that $G$ acts point-primitively on $\mathcal{D}$ as shown by Davies in \cite{Da}. In this contest, a special attention is given to the case $\lambda=2$. In a series a paper, Regueiro \cite{ORR,ORR1,ORR2,ORR3} proved that, if $\mathcal{D}$ is symmetric, then either $(v,k)=(7,4),(11,5),(16,6)$, or $v$ is a power of an odd prime and $G \leq A \Gamma L_{1}(v)$. In 2016, Liang and the fourth author \cite{LT} proved that, if $\mathcal{D}$ is non-symmetric and $G$ is point-primitive, then $G$ is an affine or an almost simple group. In each of these cases, $G$ has a unique minimal normal subgroup $T$, its socle $Soc(G)$, which is an elementary abelian group or a non-abelian simple group, respectively. Further, in the same paper Liang and Zhou and in \cite{LT1}, under the assumption of the point-primitivity of $G$ on $\mathcal{D}$, they classify $\mathcal{D}$ when $T$ is either sporadic or an altenating group, respectively. In 2020, Devillers, Liang, Praeger and Xia \cite{DLPX} showed that, if $\mathcal{D}$ is non-symmetric, then $G$ is point-primitively on $\mathcal{D}$, and hence $G$ is affine or almost simple. Moreover, they classified $\mathcal{D}$ when $T \cong PSL_{n}(q)$ for $n \geq 3$. Very recently, Liang and the first author \cite{LM} have completed the classification of the flag-transitive $2$-$(v,k,2)$ when $T$ is an elementary abelian except when $v$ is a power of a prime and $G \leq A \Gamma L_{1}(v)$, whereas Alavi et al. \cite{Alavi, ABDT} have classified $\mathcal{D}$ when $T$ is almost simple except when $T \cong PSL(2,q)$ and $\mathcal{D}$ is as follows:  
\begin{enumerate}
    \item[(I)] $\cal D$ has parameters $(v,b,r,k)=(\frac{q(q-1)}{2}, 2(q^2-1),2(q+1),\frac{q}{2})$, $q=2^{f} >8$, $f \geq 1$, $T_{x} \cong D_{2(q+1)}$ with $x$ a point of $\mathcal{D}$, and $T_B\cong Z_{2}^{f-1}$ or $Z_{2}^{f}$ with $B$ a block of $\mathcal{D}$;
    \item[(II)] $\cal D$ has parameters $(v,b,r,k)=(\frac{q(q-1)}{2}, \frac{q(q+1)}{2}, q-1,q+1)$, $q=p^{f}>5$, $p$ odd, $f \geq 1$, $T_{x} \cong D_{q+1}$ and $T_B\cong D_{q-1}$.
\end{enumerate}

A reason why the two cases are left open is that the typical group theoretical tools used to handle this type of problems become ineffective when $D$ is `close' to a classical example of $2$-design with a different $\lambda$. For instance, in  (I), the point set of $\mathcal{D}$, the block size and the group actions are those of the Witt Bose-Shrikhande linear space. As pointed out in \cite{BDD}, a model for Witt Bose-Shrikhande linear space $W(q)$, $q \geq 8$, $q$ even, can obtained by using the complementary set and the external lines of a hyperoval of $PG(2,q)$. All the previous argument motivated us to tackle the problem from a completely different perspective, namely a more geometric one involving the Desarguesian projective plane $PG(2,q)$. More precisely, we show that the open cases cannot occur by using the action of $PSL(2,q)$ in $PG(2,q)$ and the fact it preserves remarkable geometric structures such as a conic and, when $q$ is even, a hyperoval.

There is a small gap in \cite[Proposition 3.1]{ABDT} that leads to case (I), and for this reason the admissible case $T_B\cong Z_{2}^{f}$ is missed, hence we inserted it in our paper as an additional case to be analysed. Our result is the following:

\begin{theorem}\label{TH1}
Let  $\mathcal{D}$  be a  nontrivial $2$-$(v,k,2)$ design admitting a flag-transitive group $G$ of automorphisms with $Soc(G)=PSL(2,q)$, $q =p^f \geq4$, then $({\cal D}, G)$ is as in one of the lines in Table \ref{t1}.
 
\begin{table}[htbp]
\begin{center}
\caption{$2$-designs with $\lambda=2$ admitting a flag-transitive automorphims group with socle $PSL(2,q)$.}\label{t1}
\medskip
\begin{tabular}{ccccccccccc}
\hline
Line    & $v$& $b$ & $r$ & $k$ & $G$    &$G_x$ & $G_B$  & $Aut(\cal D)$  &References\\
\hline
1~&6&10&5&3&$PSL(2,5)$  &$D_{10}$&$S_3$ &$PSL(2,5)$&\cite[Theorem 1.1]{LT1} \\
2~&7&7&4&4&$PSL(2,7)$&$S_4$  &$S_4$&$PSL(2,7)$&\cite[Theorem 1]{ORR}  \\
3~&10&15&6&4&$PGL(2,5)$&$D_{12}$  &$D_8$&$S_6$&\cite[Theorem 1.1]{LT1} \\
4~&&&&&$PSL(2,9)$&$3^2:4$  &$S_4$& &\cite[Theorem 1.1]{LT1}\\
5~&11&11&5&5&$PSL(2,11)$&$A_5$  &$A_5$&$PSL(2,11)$&\cite[Theorem 1]{ORR} \\
6~&28&252&27&{7}&$PSL(2,8)$ & $D_{18}$  &$D_{14}$&$P\Gamma L(2,8)$&\cite[Proposition 3.1]{ABDT} \\
7~&&&&&$P\Gamma L(2,8)$ & $D_{18}:3$  &$7:6$&&\cite[Proposition 3.1]{ABDT}\\
8~&36&84&14&6&$PSL(2,8)$&$D_{14}$  &$S_3$&$P\Gamma L(2,8)$ &\cite[Theorem 3.10]{MF} \\
9~&&&&&$P\Gamma L(2,8)$&$7:6$  & $3 \times S_{3}$& & \cite[Theorem 3.10]{MF}\\
\hline
\end{tabular}
\end{center}
\end{table}
\end{theorem}

Some references in Table \ref{t1} are rectified with respect to those provided in \cite[Table 1]{ABDT}. For instance, the $2$-$(36,6,2)$ design as in Lines 8--9 was actually constructed in \cite[Example 3.9(1) and Theorem 3.10]{MF} and the proof is computer free.

\section{Proof of Theorem \ref{main}}
Let  $\mathcal{D}$  be a  nontrivial $2$-$(v,k,2)$ design admitting a flag-transitive automorphism group $G$ with $Soc(G)=PSL(2,q)$, $q =p^f \geq4$. Denote $Soc(G)$ by $T$, then $T \unlhd G \leq Aut(T)$, where $Aut(T) \cong P \Gamma L(2,q)$. Moreover, we denote by $x$ and $B$ any point and block of $\mathcal{D}$, respectively.

Our starting point is the following proposition, which is essentially proven in \cite[Proposition 3.1]{ABDT}. 

\begin{proposition}\label{D2(q+1)even}
Let  $\mathcal{D}$  be a  nontrivial $2$-$(v,k,2)$ design admitting a flag-transitive group $G$ of automorphisms with $Soc(G)=PSL(2,q)$, $q =p^f \geq4$, then $({\cal D}, G)$ is as in one of the lines in Table \ref{t1},or one of the following holds:
\begin{enumerate}

\item[(I)]$\cal D$ has parameters $(v,b,r,k)=(\frac{q(q-1)}{2}, 2(q^2-1),2(q+1),\frac{q}{2})$, $q=2^{f} >8$, $f \geq 1$, $T_{x} \cong D_{2(q+1)}$ and $T_B\cong Z_{2}^{f-1}$ or $Z_{2}^{f}$;
\item[(II)]$\cal D$ has parameters $(v,b,r,k)=(\frac{q(q-1)}{2}, \frac{q(q+1)}{2}, q-1,q+1)$, $q=p^{f}>5$, $p$ odd, $f \geq 1$, $T_{x} \cong D_{q+1}$ and $T_B\cong D_{q-1}$.
\end{enumerate}

\end{proposition}
{\bf{Proof.}}
The proof is that \cite[Proposition 3.1]{ABDT} except for case (I) when it is derived $T_B\cong Z_{2}^{f-1}$ using the fact that $T$ acts blocks-transitively on $\mathcal{D}$ (see \cite[Proposition 3.1(8.1)]{ABDT}). However, their argument based on the analysis of the maximal subgroups of $G$ containing $T_{B}$ still works without the assumption of on the block-transitivity of $T$ on $\mathcal{D}$ but this leads to an extra case: the block set of $\mathcal{D}$, which has size $b=2(q^{2}-1)$, is partitioned into two $T$-orbits of equal length $q^{2}-1$, and hence $T_B\cong Z_{2}^{f}$.  $\hfill\square$

\bigskip
In the sequel, we refer to the $2$-designs recorded in (I) and (II) of Proposition \ref{D2(q+1)even} as $2$-designs of type I and II, respectively. 
\bigskip

\subsection{Exclusion of the $2$-designs of type I}

In this section, we prove the following result.

\begin{theorem}\label{soon}
There are no $2$-designs of type I.    
\end{theorem}

In order to prove Theorem \ref{soon}, we need to recall the following useful facts:

\begin{enumerate}
\item[(1)] An irreducible conic $\mathcal{C}$ of $PG(2,q)$, $q=2^{f}$, is a $%
(q+1)$\emph{-arc}, namely a set of $q+1$ points no three of them collinear, by 
\cite[Lemma 7.7]{Hir}. Any line of $PG(2,q)$ is either \emph{secant}, \emph{%
tangent} or \emph{external} according as it has $2$, $1$ or $0$ points in
common with $\mathcal{C}$, respectively. The tangent lines to $\mathcal{C}$
are all concurrent to a point $N$ called \emph{nucleus} of $\mathcal{C}$ by 
\cite[Corollary 7.11]{Hir}, and $\mathcal{J}=\mathcal{C}\cup \left\{
N\right\} $ is a $(q+2)$-arc called \emph{regular} \emph{hyperoval }(see \cite[%
Section 8.4]{Hir}).

\item[(2)] The lines of $PG(2,q)$ are either secants or external to $%
\mathcal{J}$, that is they have either $2$ or $0$ points in common with $%
\mathcal{J}$. The set $\mathcal{E}$ of the external lines to $\mathcal{J}$
has size $\frac{q(q-1)}{2}$ (see \cite[Section 8.1.]{Hir}). The number of
points of $PG(2,q)\setminus \mathcal{J}$ is $q^{2}-1$ and through each point
there are exactly $\frac{q}{2}+1$ secant lines to $\mathcal{J}$ and $\frac{q%
}{2}$ external lines to $\mathcal{J}$ by \cite[Corollary 8.8]{Hir}.

\item[(3)] $PGL(3,q)$ has a unique conjugacy class of subgroups isomorphic
to $PSL(2,q)$ by \cite[Table 8.3]{BHRD} (note that $PSL(2,q) \cong \Omega (3,q)$ by \cite[Corollary 7.14]{Hir} is reducible
and not maximal in $PGL(3,q)$ when $q$ is even. Further, the irreducible conics do not arise from
polarities in this case), and each of these groups is the stabilizer in $%
PGL(3,q)$ of a suitable regular hyperoval of $PG(2,q)$. The converse is also
true as a consequence of \cite[Theorem 7.4]{Hir}. In particular, each $%
PSL(2,q)$ fixes the nucleus of its invariant hyperoval and acts $2$%
-transitively on the remaining $q+1$ points of this one  by \cite[Corollary 7.15]{Hir}.

\item[(4)] $PGL(3,q)$ has a unique conjugacy class of elements of order $2$, and if $%
\sigma $ is any of these then $\sigma $ is a $(P_{\sigma },t_{\sigma })$%
-elation of $PG(2,q)$. That is, $\sigma $ fixes each of the $q+1$ points of $%
t_{\sigma }$ including $P_{\sigma }$ and fixes setwise each of the $q+1$ lines
of $PG(2,q)$ containing $P_{\sigma }$ by \cite[Exercise IV.4.6]{HP}. No
other points or lines of $PG(2,q)$ are fixed by $\sigma $. 

\item[(5)] $PSL(2,q)$ has a unique conjugacy class of subgroups of order $q/2
$. Indeed, each of these lie in a unique Frobenius subgroup of $PSL(2,q)$ of
order $q(q-1)$.
\end{enumerate}

By (3), in the sequel, we may assume that $T$ is the copy of $PSL(2,q)$ inside $PGL(3,q)$
preserving a fixed $\mathcal{J}$ and, as pointed out in \cite[Section 2.6]{BDD}, we may identify
the point set of $\mathcal{D}$ with $\mathcal{E}$, the set of lines $%
PG(2,q)$ external to $\mathcal{J}$.

\bigskip

Let $S$ be any Sylow $2$-subgroup of $T$. Then $S$ is an elementary abelian $2$-group and $N_{T}(S)=S:K$, where $K$ is cyclic of order $q-1$. Then $S\setminus \{1\}=\left\{\sigma _{1},...,\sigma
_{q-1}\right\}$, where $o(\sigma_{i})=2$, and $K$ acts regularly on $S\setminus \{1\}$.
\begin{lemma}\label{Frob}
The following hold:
\begin{enumerate}
    \item $S$ fixes the nucleus $N$, a unique point $Q$ of $\mathcal{C}$ and acts regularly on $\mathcal{C}\setminus \left\{ Q\right\}$;
    \item $S$ fixes $t$ pointwise, where $t=NQ$, and acts semiregularly on $PG(2,q)\setminus t$;
    \item  $\sigma _{i}$ is a $(P_{\sigma_{i}},t)$-elation of $PG(2,q)$;
    \item  $C$ acts regularly on $t\setminus \left\{ N,Q\right\}=\left\{ P_{\sigma _{1}},...,P_{\sigma _{q-1}}\right\}$;
    \item If $E_{i}$ is the set of $q/2$ lines through $P_{\sigma_{i}}$ which are external to $\mathcal{J}$ (see (2)), then $S$ acts transitively on $E_{i}$ with action kernel $\left\langle \sigma _{i} \right\rangle$.
\end{enumerate}
\end{lemma}
\begin{proof}
The group $G$ fixes $N$ by (3), so does $S$. Moreover, $S$ fixes a unique point $Q$ of $\mathcal{C}$ and acts
regularly on $\mathcal{C}\setminus \left\{ Q\right\}$ since $T$ acts $2$-transitively on the $q+1$ points of $\mathcal{C}$, which is (i). Thus $S$ preserves $t$, where $t=NQ$. The actions of $S$ on $\mathcal{C}$ and on the set of $q+1$ lines through $N$ are equivalent, as these are tangents to $\mathcal{C}$ by (1), hence $S$ and acts semiregularly on $%
PG(2,q)\setminus t$.

Each $\sigma _{i}$ is a $(P_{\sigma
_{i}},t_{\sigma _{i}})$-elation of $PG(2,q)$ for each $i=1,...,q-1 $ by (4). Then $t_{\sigma _{i}}=t
$ since $S$, and hence $\sigma _{i}$, fixes $N$ and $Q$. Thus $\sigma _{i}$ fixes $t$ pointwise for each $i=1,...,q-1$, and hence $S$
fixes $t$ pointwise. This proves (ii) and (iii).

If $P_{\sigma _{i}}= N$, then $\sigma_{i}$ preserves each of the $q+1$ lines through $N$ by (4), and hence $\sigma _{i}$ fixes $\mathcal{C}$ pointwise since $\sigma _{i}$ preserve $\mathcal{C}$ and each line through $N$ is tangent to $\mathcal{C}$. This contradicts (i), hence $P_{\sigma _{i}} \neq N$.

If $P_{\sigma _{i}}= Q$, then $\sigma_{i}$ preserves each of the $q$ secants to $\mathcal{C}$ through $Q$ by (4), and hence $\sigma _{i}$ fixes $\mathcal{C}$ pointwise, and we again reach a contradiction. Thus, $P_{\sigma _{i}} \neq Q$.
It follows from (ii) that, $t$ is the set of points of $PG(2,q)$ fixed by $S$. Then $K$ preserves $t$ since $K$ normalizes $S$. Further $K$ fixes $N$ since $T$ does it, and $K$ fixes $Q$ since $\{Q\}=t \cap \mathcal{C}$ and $K$ preserves both $t$ and $\mathcal{C}$.

If there is a non-trivial element $%
\psi $ of $K$ fixing a point $P$ on $t\setminus \left\{ N,Q\right\} $, then $%
\psi $ fixes one of the $\frac{q}{2}$ external lines to $\mathcal{J}$
containing $P$ by (ii), say $x$. So $\psi \in T_{x}$, whereas $T_{x}\cong D_{2(q+1)}$%
. Thus $K$ acts regularly on $t\setminus \left\{ N,Q\right\} $. Then $t\setminus \left\{ N,Q\right\} =\left\{ P_{\sigma _{1}},...,P_{\sigma _{q-1}}\right\}$ since we have seen that $\left\{ P_{\sigma _{1}},...,P_{\sigma _{q-1}}\right\}$ is a subset of $t\setminus \left\{ N,Q\right\}$ and $K$ acts regularly on $ S\setminus \{1\}=\left\{\sigma _{1},...,\sigma_{q-1}\right\}$. This proves (iv).

Since $S$ preserves $\mathcal{C}$, fixes $N$, and fixes each $P_{\sigma_{i}}$ by (ii), it follows that $S$ permutes the set $E_{i}$ of $q/2$ lines through $P_{\sigma_{i}}$ which are external to $\mathcal{J}$. Now, $\sigma _{i}$ preserves each line containing $P_{\sigma _{i}}$ by (4), hence $\left\langle \sigma _{i} \right\rangle$ lies in the action kernel $A$ of $S$ on $E_{i}$. On the other hand, if $y \in E_{i}$, then  $A \leq S_{y} =\left\langle \sigma _{i} \right\rangle $ since  $S_{y} \leq T_{y}\cong D_{2(q+1)}$ with $q$ even. Thus $A=S_{y}=\left\langle \sigma _{i} \right\rangle$ and $S$ acts transitively on $E_{i}$. This proves (v).  $\hfill\square$
\end{proof}

\bigskip

\begin{proof}[of Theorem \ref{soon}]\,\,
Let $B$ be any block of $\mathcal{D}$. Then the order of $T_{B}$ is either $q/2$ or $q$, and in both cases we may assume that $T_{B}$ is a fixed subgroup of $S$ by (5). 

Suppose that $\left\vert T_{B} \right\vert=q/2$. Then $T$ acts flag-transitively on $\mathcal{B}$. Hence $B=\ell ^{T_{B}}$, where $\ell $ is a suitable external line to $%
\mathcal{J}$. Without loss of generality, we may assume that $T_{B}\setminus \left\{ 1\right\}
=\left\{ \sigma _{1},...,\sigma _{q/2-1}\right\}$.

Let $R$ be the
intersection point of $\ell $ with $t$. Then $R\neq Q,N$ since $Q\in 
\mathcal{C}$, $N$ is the nucleus of $\mathcal{J}$ and $\ell$ is external to $\mathcal{J}$
(see \cite[Corollary 8.8]{Hir}). Then $R=P_{\sigma _{i}}$ for a unique $i\in
\left\{ 1,...,q-1\right\} $ since $t\setminus \mathcal{J}=\left\{ P_{\sigma
_{1}},...,P_{\sigma _{q-1}}\right\} $, and hence $\ell \in E_{i}$, where $E_{i}$ is the set of $q/2$ lines through $P_{\sigma_{i}}$ which are external to $\mathcal{J}$ (see Lemma \ref{Frob}(v)). Then $B \subseteq E_{i}$ since $T_{B}$ fixes $P_{i}$, being $T_{B}$ a subgroup of $S$. Actually, $B=E_{i}$ since $k=q/2$. Then $T_{\ell, B} = \left \langle \sigma _{i}\right\rangle$ by Lemma \ref{Frob}(v), and hence $k=\left\vert B\right\vert =\left\vert \ell ^{T_{B}}\right\vert =q/4$, a contradiction.

Suppose that $\left\vert T_{B} \right\vert=q$. Then $\left\vert T:T_{B}\right\vert=q^{2}-1$, and hence $\left\vert G:G_{B}\right\vert=2\left\vert T:T_{B}\right\vert$. Then $\mathcal{B}=B^{G}=B^{T}\cup C^{T}$ with $B$ and $C$ blocks of $\mathcal{D}$ both preserved by $T_{B}$. Then $B=\ell ^{T_{B}}$ and $C=m ^{T_{B}}$ with $\ell$ and $m$ suitable external lines to $\mathcal{J}$. Now, arguing as above, one has $B=E_{i}$ and $C=E_{j}$, where $E_{i}$ and $E_{j}$ are the set of $q/2$ external lines to $\mathcal{J}$ through $P_{\sigma_{i}}$ and $P_{\sigma_{j}}$, respectively. Moreover, $P_{\sigma_{i}}\neq P_{\sigma_{j}}$ since $E_{i}=B\neq C=E_{j}$ by (2). 

Let $\varphi$ be the unique element of $K$ such that $P_{\sigma_{i}}^{\varphi}=P_{\sigma_{j}}$ by Lemma \ref{Frob}(iv), then $E_{i}^{\varphi}=E_{j}$ since $\varphi$ preserves $\mathcal{J}$, and $E_{i}$ and $E_{j}$ are the set of all (namely $q/2$ each) external lines to $\mathcal{J}$ through $P_{\sigma_{i}}$ and $P_{\sigma_{j}}$, respectively. That is $B^{\varphi}=C$ and hence $\mathcal{B}=B^{G}=B^{T}\cup C^{T}=B^{T}$, whereas $\left\vert G:G_{B}\right\vert=2\left\vert T:T_{B}\right\vert$, a contradiction. This completes the proof.  $\hfill\square$
\end{proof}

\begin{remark}
It is not difficult to see that $(\mathcal{P}, B^{G})$, where $\mathcal{P}$ is the set of external lines to $\mathcal{J}$, and $B=\ell^{S}$ with $\ell \in \mathcal{P}$ is an isomorphic copy of the Witt Bose-Shrikhande linear space $W(q)$, $q \geq 8$ even.

\end{remark}

\subsection{Exclusion of the $2$-designs of type II}
In this section, we prove the following theorem, thus completing the proof of Theorem \ref{TH1}. 

\begin{theorem}\label{main} 
$\mathcal{D}$ is not a $2$-design of type II.
\end{theorem}

\medskip
Throughout the remainder of the paper, we denote by $X$ the copy of $PGL(2,q)$ contained in $Aut(T)$.

\begin{lemma}\label{mom} If $q\equiv -1\pmod{4}$, then one of the following holds:
\begin{enumerate}
    \item $T$ acts flag-transitively on $\mathcal{D}$;
    \item $\mathcal{D}$ admits $X$ as a flag-transitive automorphism group,  and each block $B$ of $\mathcal{D}$ is a union of two $T_{B}$-orbits of length $\frac{q-1}{2}$.
\end{enumerate}
\end{lemma}

\begin{proof}
Let $(x,B)$ any flag of $\mathcal{D}$. Then $\left\vert (x,B)^{G}\right\vert
=\frac{q(q^{2}-1)}{2}$ by the flag-transitivity of $G$ on $\mathcal{D}$. We know that $%
T\trianglelefteq G$ and $T_{x,B}\leq Z_{2}$. If $T_{x,B}=1$ then $\left\vert G:G_{x,B}\right\vert =\left\vert
T:T_{x,B}\right\vert $. Thus $T$ acts flag-transitively on $\mathcal{D}$, and we obtain (1).

If $T_{x,B}=Z_{2}$, then $\left\vert G:G_{x,B}\right\vert =2\left\vert
T:T_{x,B}\right\vert $ and $B$ is a union of two $T_{B}$-orbits of length $%
\frac{q-1}{2}$. Then $\left\vert G:T\right\vert =2\left\vert
G_{x,B}:T_{x,B}\right\vert $, and hence the order of $G/T$ is even.

If $X\nleq G$, then $G \cap X=T$ and hence $G/T \cong GX/X \leq Out(T)/X \cong Z_{f}$, where $f=\log _{p}(q)$. Then $f$ is even since $G/T$ has even order, and hence $q\equiv 1\pmod{4}$, a contradiction. Thus $X\leq G$. 

Now, $Z_{2}= T_{x,B}\leq X_{x,B}\leq Z_{2}\times Z_{2}$. If $X_{x,B}\cong
Z_{2}\times Z_{2}$ and so $G\neq X$ and $\left\vert G:X\right\vert
=2\left\vert G_{x,B}:X_{x,B}\right\vert $ since $X\vartriangleleft G$. Thus $%
f$ is even since $G/X \leq  Z_{f}$, and again $q\equiv 1\pmod{4}$, which is not the case. Thus $X_{x,B}=Z_{2}$, and hence $X$ acts flag-transitively on $\mathcal{D}$, which is (2). $\hfill\square$
\end{proof}

\subsubsection{Identification of the blocks of $\mathcal{D}$}

Since $PGL (3,q)$ has one conjugacy class of subgroups
isomorphic to $PSL(2,q)$ and each of these preserves a unique irreducible
conic of $PG(2,q)$, we may assume that $T\cong PSL(2,q)$ is the stabilizer of the conic%
\[
\mathcal{C}:X_{0}X_{2}-X_{1}^{2}=0. 
\]
by \cite[Corollary 7.14]{Hir}. More details on conics can be found in \cite[Section 7.2]{Hir}. A point $P$ of $PG(2,q)$ is \emph{internal} or \emph{external} to $%
\mathcal{C}$ according as it lies on $0$ or $2$ tangents to $\mathcal{C}$ by \cite[Section 8.2]{Hir}.
Let $I$ and $E$ be the set of internal and external points to $\mathcal{C}$.
Then $\left\vert I\right\vert =\frac{q(q-1)}{2}$ and $\left\vert
E\right\vert =\frac{q(q+1)}{2}$ by \cite[Section 8.2]{Hir}. In the sequel, $P$ will be called $I$\emph{%
-point of }$\mathcal{C}$, or $E$\emph{-point of }$\mathcal{C}$ according as
it lies in $I$ or $E$, respectively.

Note that, a point $P$ of $PG(2,q)$ is a $I$-point (resp. $E$-point) of $%
\mathcal{C}$ if and only if $P^{\pi }$ is a external (resp. secant) to $%
\mathcal{C}$ by \cite[Theorem 8.16]{Hir}, where 
\[
\pi:
\left( 
\begin{array}{ccc}
Y_{0},&Y_{1},&Y_{2}
\end{array}%
\right)
\longmapsto \left( 
\begin{array}{ccc}
Y_{0},&Y_{1},&Y_{2}
\end{array}%
\right) \left( 
\begin{array}{ccc}
0 & 0 & 1/2 \\ 
0 & -1 & 0 \\ 
1/2 & 0 & 0%
\end{array}%
\right)
\left(\begin{array}{c}
X_{0}\\
X_{1}\\
X_{2} 
\end{array}%
\right)
=0
\]%
is the polarity defined by $\mathcal{C}$.

\bigskip

Le $\chi $ be the quadratic character of $GF(q)$, then $\chi (z)$ is $0$, $1$
or $-1$ according as $z$ is either $0$ or is a square, or a non-square of $GF(q)$, respectively.
Let $Q_{+1}$ and $Q_{-1}$ be the set of squares and non-squares of $GF(q)^{\ast}$, respectively,
then $\left\vert Q_{+1}\right\vert =\left\vert Q_{-1}\right\vert =\frac{q-1}{2}
$.

\bigskip

It is well known that, $T$ acts transitively on $I$ and on $E$, and the
stabilizer of a point is isomorphic to $D_{q+1}$ or $D_{q-1}$, respectively.
On the other hand, $T$ has a unique conjugacy class of subgroups isomorphic to $%
D_{q+1}$. Hence, we may identify the point set $\mathcal{P}$ of $\mathcal{D}
$ with $I$. Hence, any block of $\mathcal{D}$ is a suitable subset of size $q-1$ of $%
I $.

\bigskip 
Note that, $T$ is the group consisting of the elements represented by%
\[
\left( 
\begin{array}{ccc}
a^{2} & ab & b^{2} \\ 
2ac & ad+bc & 2bd \\ 
c^{2} & cd & d^{2}%
\end{array}%
\right) 
\]
where $ab-bd\in Q_{+}$ by \cite[Corollary 7.14 and its proof]{Hir}. As note
above, $T$ has a unique conjugacy class of subgroups isomorphic to $D_{q-1}$. Therefore, we may assume that $T_{B}=\left\langle \alpha ^{2},\beta
\right\rangle$ for some block $B$ of $\mathcal{D}$, where $\alpha $ and $\beta $ are respectively represented
by 
\[
\alpha=\left( 
\begin{array}{ccc}
\omega & 0 & 0 \\ 
0 & 1  & 0 \\ 
0 & 0 & \omega^{-1}%
\end{array}%
\right) \text{ and } \beta=\left( 
\begin{array}{ccc}
0 & 0 & 1 \\ 
0 & -1 & 0 \\ 
1 & 0 & 0%
\end{array}%
\right) 
\]%
where $\omega $ is a primitive element of $GF(q)^{\ast}$.

Note that, $T_{B}$
preserves a pencil of bitangent conics in $O=(0,0,1)$ and $P_{\infty
}=(1,0,0)$. The conics in the pencil are:

\begin{enumerate}
\item The irreducible conics $\mathcal{C}_{h}:X_{0}X_{2}-hX_{1}^{2}=0$ for $%
h\in GF(q)^{\ast }$ (here $\mathcal{C}_{1}=\mathcal{C}$);

\item The simply degenerate conic $\mathcal{C}_{0}:X_{0}X_{2}=0$;

\item The doubly degenerate conic $\mathcal{C}_{\infty }:X_{1}^{2}=0$.
\end{enumerate}

\bigskip

In the sequel, for $h\in GF(q)\cup \left\{ \infty \right\} $ we denote the set $\mathcal{C}_{h}\setminus \left\{ O,P_{\infty
}\right\} $ by $\mathcal{C}_{h}^{\ast }$. Hence, $\mathcal{C}_{h}^{\ast }$ is a $(q-1)$-arc of $PG(2,q)$. Further, the set $\{\mathcal{C}_{h}^{\ast }:h\in GF(q)\cup \left\{ \infty \right\}\}$ is a partition of $PG(2,q)\setminus \left\{ O,P_{\infty }\right\} $.

\begin{proposition}\label{orbit}
One of the following holds:
\begin{enumerate}
    \item $q\equiv -1 \pmod{4}$ and one of the following holds:
    \begin{enumerate}
        \item The $T_{B}$-orbits of $I$-points of $\mathcal{C}$ of length $q-1$ are the $(q-1)$-arcs $\mathcal{C}_{h}^{\ast}$ with $h \in (1+Q_{+1})\cap Q_{+1}$.
        \item The $X_{B}$-orbits of $I$-points of $\mathcal{C}$ of length $q-1$, which are a union of two $T_{B}$-orbits each of length $\frac{q-1}{2}$, are the $(q-1)$-arcs $\mathcal{C}_{h}^{\ast}$ with $h \in (1+Q_{+1})\cap Q_{-1}$.
    \end{enumerate}
    \item $q\equiv 1 \pmod{4}$ and the following hold: 
  \begin{enumerate}
      \item The $T_{B}$-orbits of $I$-points of $\mathcal{C}$ of length $q-1$ are the $(q-1)$-arcs $\mathcal{C}_{h}^{\ast}$ with $h \in (1+Q_{-1})\cap Q_{-1}$.
      \item The $T_{B}$-orbits of $I$-points of $\mathcal{C}$ of length $\frac{q-1}{2}$ are the following:
\begin{eqnarray*}
\mathcal{O}_{\infty } &=&\left\{ (\mu,0,1):\mu\in Q_{-1}\right\}, \\
\mathcal{O}_{h,+1} &=&\left\{ \left( h\mu ^{2},\mu ,1\right) :\mu \in
Q_{+1}\right\}, \\
\mathcal{O}_{h,-1} &=&\left\{ \left( h\mu ^{2},\mu ,1\right) :\mu \in
Q_{-1}\right\}
\end{eqnarray*}%
where $h\in  (1+Q_{-1})\cap Q_{+1}$.
  \end{enumerate}  
\end{enumerate}
\end{proposition}

\begin{proof}
Let $A=(\mu,0,1)$ be a point of $\mathcal{C}_{\infty }$. Then $A^{\pi
}:X_{0}+\mu X_{2}=0$. Hence 
\[
A\in I\Longleftrightarrow \left\vert A^{\pi }\cap \mathcal{C}\right\vert
=0\Longleftrightarrow \mu Z^{2}+1=0\text{ has }0\text{ solutions,} 
\]%
where $Z=X_{2}/X_{1}$. This occurs if and only if $-\mu^{-1} \in Q_{-1}$, and hence $\mu$
lies in $Q_{-1}$ or $Q_{+1}$ according as $q \equiv 1 \pmod{4}$ or $q \equiv {-1} \pmod{4}$, respectively. Thus, 
\[
\mathcal{O}_{\infty }=\left\{ (\mu,0,1):\mu\in Q_{-\varepsilon}\right\} 
\]%
is a $T_{B}$-orbit of length $\frac{q-1}{2}$ consisting of $I$ of $\mathcal{C}$ for $q \equiv \varepsilon \pmod{4}$, where $\varepsilon= \pm 1$. 

For $h\in GF(q)^{\ast }\setminus \left\{ 1\right\} $, one has $\mathcal{C}%
_{h}^{\ast }=\mathcal{O}_{h,+1}\cup \mathcal{O}_{h,-1}$, where 
\begin{eqnarray*}
\mathcal{O}_{h,+1} &=&\left\{ \left( h\mu ^{2},\mu ,1\right) :\mu \in
Q_{+1}\right\}, \\
\mathcal{O}_{h,-1} &=&\left\{ \left( h\mu ^{2},\mu ,1\right) :\mu \in
Q_{-1}\right\}
\end{eqnarray*}%
are two orbits, each of length $\frac{q-1}{2}$, under the cyclic subgroup of $T_{B}$ of order $\frac{q-1}{2}$.

Let \ $P_{\mu }=\left( h\mu ^{2},\mu ,1\right)$
be any point of $\mathcal{C}_{h}^{\ast }$ with $h\in GF(q)^{\ast }\setminus
\left\{ 1\right\} $, then $P_{\mu }$ is a $I$-point of $\mathcal{C}$ if and only if $P_{\mu
}^{\pi }$ is an external line to $\mathcal{C}$, where 
\[
P_{\mu }^{\pi }:X_{0}-2\mu X_{1} +h\mu ^{2}X_{2}=0 \textit{.}
\]%
Hence, 
\[
P_{\mu }\in I\Longleftrightarrow \left\vert P_{\mu }^{\pi }\cap \mathcal{C}%
\right\vert =0\Longleftrightarrow h\mu ^{2}Z^{2}-2\mu Z+1=0\text{ has }0\text{
solutions} 
\]%
where $Z=X_{2}/X_{1}$. This happens if and only if $4\mu ^{2}-4\mu ^{2}h\in
Q_{-}$, that is, if and only if $h-1\in Q_{-\varepsilon}$. Further, since $P_{\mu }^{\beta
}=\left( h\left( -1/h\mu\right)^{2},-1/h\mu
,1\right) $, it follows that $\mathcal{O}_{h,j}^{\beta }=\mathcal{O}_{h,j}$, $j=\pm1$, if
and only if $\chi(-h\mu)$=$\chi(\mu)$, that is, only for $\chi(h)=\chi(-1)$ since $\mu \neq 0$.

Assume that $q \equiv 1 \pmod{4}$. The $T_{B}$-orbits of $I$-points of $\mathcal{C}$ of length $\frac{q-1}{2}$ are  $\mathcal{O}_{\infty }$ and $\mathcal{O}_{h,j}$ with $h \in (1+Q_{-1})\cap Q_{+1}$ and $j= \pm 1$. Hence, the $T_{B}$-orbits of $I$-points of $\mathcal{C}$ of length $q-1$ are the $(q-1)$-arcs $\mathcal{C}_{h}^{\ast}$ with $h \in (1+Q_{-1})\cap Q_{-1}$.

Assume that $q \equiv {-1}\pmod{4}$. Arguing as above, it is easy to see that the $T_{B}$-orbits of $I$-points of $\mathcal{C}$ of length $q-1$ are the $(q-1)$-arcs $\mathcal{C}_{h}^{\ast}$ with $h \in (1+Q_{+1}) \cap Q_{+1}$, whereas the $X_{B}$-orbits of $I$-points of $\mathcal{C}$ of length $q-1$, which are a union of two $T_{B}$-orbits each of length $\frac{q-1}{2}$, are the $(q-1)$-arcs $\mathcal{C}_{h}^{\ast}$ with $h \in (1+Q_{+1})\cap Q_{-1}$ since $X_{B}=\left\langle \alpha, \beta \right\rangle$ $\hfill\square$.
\end{proof}

\begin{corollary}
\label{Elle}The following hold:
\begin{enumerate} 
    \item  If $q\equiv -1 \pmod{4}$, then $\left\vert (1+Q_{+1}) \cap Q_{+1} \right\vert= \frac{q-3}{4}$ and $\left\vert (1+Q_{+1}) \cap Q_{-1} \right\vert= \frac{q+1}{4}$;
    \item If $q\equiv 1 \pmod{4}$, then $\left\vert (1+Q_{-1}) \cap Q_{-1} \right\vert= \left\vert (1+Q_{-1}) \cap Q_{+1}  \right\vert=\frac{q-1}{4}$.
\end{enumerate}
\end{corollary}

\begin{proof}
 If $q\equiv -1 \pmod{4}$, then $\left\vert (1+Q_{+1}) \cap Q_{+1} \right\vert= \frac{q-3}{4}$ since $(1+Q_{+1}) \cap Q_{+1}$ is the number of common points of $1+Q_{+1}$ and $Q_{+1}$ regarded as blocks of the development of the Paley-Hadamard $\left(q,\frac{q-1}{2}, \frac{q-3}{4}\right)$-difference set by \cite[Theorem VI.1.12]{BJL}. Consequently, one has $$\left\vert (1+Q_{+1}) \cap Q_{-1} \right\vert=\frac{q-1}{2}- \frac{q-3}{4}=\frac{q+1}{4} \text{.}$$ 

If $q\equiv 1 \pmod{4}$, then $\left\vert \left( 1+Q_{+1}\right)
\cap Q_{+1}\right\vert =\frac{q-5}{4}$ since it is the number of common
neighbours of $0$ and $1$ in the Paley graph by \cite[Example III.10.15]{BJL}. Then 
$$
\left\vert \left( 1+Q_{-1}\right) \cap Q_{+1}
\right\vert=\left\vert Q_{+1}\setminus \left\{ 1\right\} \right\vert-\left\vert \left( 1+Q_{+1}\right) \cap
Q_{+1}\right\vert=\frac{q-3}{2}-\frac{q-5}{4}=\frac{q-1}{4}\text{,}
$$
and hence
$$
\left\vert \left( 1+Q_{-1}\right) \cap Q_{-1} \right\vert=\frac{q-1}{2}-\left\vert \left( 1+Q_{-}\right) \cap Q_{+}
\right\vert=\frac{q-1}{2}-\frac{q-1}{4}=\frac{q-1}{4}\text{.}
$$
 $\hfill\square$
\end{proof}

\begin{lemma}\label{BOO}Let $B$ be any block of $\mathcal{D}$. Then one of the following holds:
\begin{enumerate}
    \item[(1)] $q\equiv -1 \pmod{4}$ and $B=\mathcal{C}_{h}^{\ast}$ for some $h \in 1+Q_{+1}$;
     \item[(2)] $q\equiv 1 \pmod{4}$ and one of the following holds:
     \begin{enumerate}
         \item $B=\mathcal{C}_{h}^{\ast}$ for some $h \in 1+Q_{-1}$;
         \item $B=\mathcal{O}_{h,i}\cup \mathcal{O}_{h^{p^{m}},j}$ for some $h\in  (1+Q_{-1})\cap Q_{+1}$, $h^{p^{m}}\neq h$ and $h^{p^{2m}}=h$, $1 \leq m \leq \log_{p}(q)/2$ and $i,j\in \left\{ -1,+1\right\} $. In particular, $\log_{p}(q)$ is even.
     \end{enumerate}    
\end{enumerate}
\end{lemma}

\begin{proof}
Let $B$ be any block of $\mathcal{D}$. Assume that $q\equiv -1 \pmod{4}$. By Lemma \ref{mom}, either $T$ acts flag-transitively on $\mathcal{D}$, $B$ is a $T_{B}$-orbit of length $q-1$, and hence $B=\mathcal{C}_{h}^{\ast}$ for some $h \in (1+Q_{+1})\cap Q_{+1}$ by Proposition \ref{orbit}(i.a), or $T$ does not act flag-transitively on $\mathcal{D}$, $X$ acts flag-transitively on $\mathcal{D}$, $B$ is a $X_{B}$-orbit of length $q-1$, and hence $B=\mathcal{C}_{h}^{\ast}$ with $h \in (1+Q_{+1})\cap Q_{-1}$ by Proposition \ref{orbit}(i.b). This proves (1).

Assume that $q\equiv 1 \pmod{4}$. If $T$ acts flag-transitively on $\mathcal{D}$, then $B$ is a $T_{B}$-orbit of length $q-1$, and hence $B=\mathcal{C}_{h}^{\ast}$ for some $h \in (1+Q_{-1})\cap Q_{-1}$ by Proposition \ref{orbit}(ii.a), and we obtain (2.a) in this case.

If $T$ does acts flag-transitively on $\mathcal{D}$, then $B$ is a union of two $B$-orbits of length $\frac{q-1}{2}$. If $\mathcal{O}_{\infty }\subset B$. Then $B=\mathcal{O}_{\infty
}\cup \mathcal{O}_{h_{0},j}$ for some $h_{0}\in (1+Q_{-1})\cap Q_{+1}$ and $j\in \left\{
-1,+1\right\} $ by Lemma \ref{orbit}(ii.b). Note that, $\frac{q-1}{2}>2$ since $q>5$.
Now, $G_{B}$ acts transitively on $B$ and $G_{B}\leq P\Gamma L(2,q)$. On
the other hand $G<P\Gamma L(3,q)$, hence the elements of $G_{B}$ are
collineations of $PG(2,q)$ and these do not map subsets of lines of size $%
\frac{q-1}{2}>2$ onto subsets of conics, being these ones arcs. Thus, $\mathcal{O}_{\infty
}\not\subset B$. Therefore $B=\mathcal{O}_{h_{1},i}\cup \mathcal{O}_{h_{2},j}$
for some $h_{1},h_{2}\in  (1+Q_{-1})\cap Q_{+1}$ and $i,j\in \left\{-1,+1\right\} $.

Assume that $h_{1}=h_{2}$. Then $i\neq j$ since $k=q-1$, and hence $B=%
\mathcal{C}_{h_{1}}^{\ast }$ with $h_{1}\in  (1+Q_{-1})\cap Q_{+1}$, and we again obtain (2.a).

Assume that $h_{1} \neq h_{2}$. The group $T$ has one conjugacy class of subgroups isomorphic to $T_{B}$,
hence $G_{B}=N_{G}(T_{B})\leq N_{P\Gamma L(2,q)}(T_{B})=\left\langle \beta
,\alpha ,\sigma \right\rangle $, where $\sigma :\left(
X_{0},X_{1},X_{2}\right) \longmapsto \left(
X_{0}^{p},X_{1}^{p},X_{2}^{p}\right) $. Now, $\mathcal{O}_{h_{s},i}^{\alpha^
t}=\mathcal{O}_{h_{s},j}$ with $j\neq i$ if $t$ is odd, and $\mathcal{O}%
_{h_{s},i}^{\alpha ^{t}}=\mathcal{O}_{h_{s},j}$ with $j=i$ if $t$ is even. Further, $\mathcal{O}_{h_{s},i}^{\sigma ^{m}}=\mathcal{O}%
_{h_{s}^{p^{m}},i}$. Hence, $\mathcal{O}_{h_{1},i}^{\beta ^{w}\alpha
^{t}\sigma ^{m}}=\mathcal{O}_{h_{1},i}$ if and only if $t$ is even and $%
h_{1}\in GF(p^d)$, where $d=\gcd (f,m)$  and $q=p^f$. Furthermore, we have
\[
\mathcal{O}_{h_{1},i}^{\beta ^{w}\alpha ^{t}\sigma ^{m}}=\mathcal{O}%
_{h_{2},j}\Longleftrightarrow \left\{ 
\begin{array}{cc}
t\text{ is odd and }h_{2}=h_{1}^{p^{m}}\text{ }, & i\neq j, \\ 
t\text{ is even and }h_{2}=h_{1}^{p^{m}}, & i=j.%
\end{array}%
\right. 
\]%
Suppose that $i\neq j$. Then there is $\beta ^{w}\alpha
^{t}\sigma ^{m}\in G_{B}$ such that $\mathcal{O}_{h_{1},i}^{\beta ^{w}\alpha
^{t}\sigma ^{m}}=\mathcal{O}_{h_{2},j}$ and $\mathcal{O}_{h_{2},j}^{\beta
^{w}\alpha ^{t}\sigma ^{m}}=\mathcal{O}_{h_{1},i}$. Hence, $t$ is odd, $%
h_{2}=h_{1}^{p^{m}}$ and $h_{1}=h_{2}^{p^{m}}$. Therefore, $f=\log_{p}(q)$ is even, $%
h_{s}^{p^{2m}}=h_{s}$ for $s=1,2$.\\
Suppose that $i=j$. Then there is $\beta
^{w}\alpha ^{t}\sigma ^{m}\in G_{B}$ such that $\mathcal{O}_{h_{1},i}^{\beta
^{w}\alpha ^{t}\sigma ^{m}}=\mathcal{O}_{h_{2},i}$ and $\mathcal{O}%
_{h_{2},j}^{\beta ^{w}\alpha ^{t}\sigma ^{m}}=\mathcal{O}_{h_{1},i}$. Hence, 
$t$ is even, $h_{2}=h_{1}^{p^{m}}$ and $h_{1}=h_{2}^{p^{m}}$. Therefore, $f=\log_{p}(q)$ is even, $%
h_{s}^{p^{2m}}=h_{s}$ for $s=1,2$ also in this case. Thus, we obtain (2.b). $\hfill\square$
\end{proof}

\bigskip

Let $c,\xi \in GF(q)$ with $\xi $ fixed, $\xi \neq 0$, and $\xi ^{2}\neq -1$ when $q \equiv 1 \pmod{4}$,
and let $\tau_{\xi} $, $\gamma _{c}$ be the elements of $T$ represented, up to a non-zero element of $GF(q)$, by the
matrices%
\begin{equation}\label{rep1}
\left( 
\begin{array}{ccc}
1 & \xi & \xi ^{2} \\ 
2\xi & -1+\xi ^{2} & -2\xi \\ 
\xi ^{2} & -\xi & 1%
\end{array}%
\right) \text{ and } \left( 
\begin{array}{ccc}
1 & c & c^{2} \\ 
0 & 1 & 2c \\ 
0 & 0 & 1%
\end{array}%
\right) 
\end{equation}%
respectively. Then $W=\left\{ \gamma _{c}:c\in GF(q)\right\} $ is the Sylow $%
p$-subgroup of $T$ fixing $O=(0,0,1)$. Then $\tau _{\xi }$ has order $2$,
and $\tau _{\xi }\gamma _{c}\alpha ^{2u}$ is represented by the matrix
\begin{equation}\label{rep2}
\left( 
\begin{array}{ccc}
\omega ^{2u} & c+\xi & \frac{1}{\omega ^{2u}}\left( c+\xi \right) ^{2} \\ 
2\xi \omega ^{2u} & \xi ^{2}+2c\xi -1 & \frac{2}{\omega ^{2u}}\left( c\xi
-1\right) \left( c+\xi \right) \\ 
\xi ^{2}\omega ^{2u} & \xi \left( c\xi -1\right) & \frac{1}{\omega ^{2u}}%
\left( c\xi -1\right) ^{2}%
\end{array}%
\right) \text{.} 
\end{equation}

\medskip

\begin{lemma}
\label{LSR1}The following hold:

\begin{enumerate}
\item If $T_{B}\gamma _{c_{1}}=T_{B}\gamma _{c_{2}}$, then $c_{1}= c_{2}$;

\item $T_{B}\gamma _{c_{1}} \neq T_{B}\tau _{\xi }\gamma _{c_{2}}\alpha ^{2u}$;

\item If $T_{B}\tau _{\xi }\gamma _{c_{1}}\alpha ^{2u_{1}}=T_{B}\tau _{\xi
}\gamma _{c_{2}}\alpha ^{2u_{2}}$ and $(c_{2},u_{2})\neq (c_{1},u_{1})$, then $q \equiv 1 \pmod{4}$ and 
$(c_{2},u_{2})=\left( \xi ^{-1}-\xi -c_{1},u_{1}+\frac{q-1}{4}\right) $.

\item If $q \equiv 1 \pmod{4}$, $T_{B}\tau _{\xi _{1}}\gamma _{c_{1}}\alpha ^{2u_{1}}=T_{B}\tau
_{\xi _{2}}\gamma _{c_{2}}\alpha ^{2u_{2}}$ with $\xi_{1} \neq \xi_{2}$, then $\chi (\xi _{1}\xi _{2})=1$.
\end{enumerate}
\end{lemma}

\begin{proof}
If $T_{B}\gamma _{c_{1}}=T_{B}\gamma _{c_{2}}$, then $\gamma
_{c_{1}-c_{2}}\in T_{B}$ and hence $\gamma _{c_{1}-c_{2}}=\beta ^{m}\alpha
^{2s}$ for some integers $m$ and $s$. Thus $\gamma_{c_{1}-c_{2}}\alpha
^{-2s}=\beta ^{m}$. Since $\gamma_{c_{1}-c_{2}}\alpha
^{-2s}$ fixes $O$,
whereas $\beta ^{m}$ does not, unless $\beta ^{m}=1$, it follows that $\beta
^{m}=1$, $\alpha ^{-2s}\gamma_{c_{1}-c_{2}}=1$, and hence $\alpha ^{2s}=1$
and $c_{1}=c_{2}$. This proves (i).

Assume that $T_{B}\gamma _{c_{1}}=T_{B}\tau_{\xi} \gamma _{c_{2}}\alpha ^{2u}$.
Then $\gamma _{c_{2}}\alpha ^{2u}\gamma _{c_{1}}^{-1}=\tau_{\xi} \beta ^{m}\alpha
^{2s}$. Then $\gamma _{c_{2}}\alpha ^{2u}\gamma _{c_{1}}^{-1}$ fixes $O$,
whereas $O^{\tau_{\xi} \beta ^{m}\alpha ^{2s}}$ is $(\xi ^{2}\omega ^{4s},-\xi \omega ^{2s},1)$ or $(\xi
^{-2}\omega ^{4s},\xi ^{-1} \omega ^{2s},1)$ according as $m=0$ or $1$, respectively, and we obtain $%
T_{B}\gamma _{c_{1}}\neq T_{B}\tau_{\xi} \gamma _{c_{2}}\alpha ^{2u}$. This proves (ii).

Assume that $T_{B}\tau_{\xi} \gamma _{c_{1}}\alpha ^{2u_{1}}=T_{B}\tau_{\xi} \gamma
_{c_{2}}\alpha ^{2u_{2}}$. Then $\gamma _{c_{1}}\alpha
^{2(u_{1}-u_{2})}\gamma _{c_{2}}^{-1}=\tau_{\xi} \beta ^{m}\alpha ^{2s}\tau_{\xi} $.
Then $\gamma _{c_{1}}\alpha ^{2(u_{1}-u_{2})}\gamma _{c_{2}}^{-1}$ is
represented by 
\[
M=\left( 
\begin{array}{ccc}
\omega ^{2u_{1}-2u_{2}} & c_{1}-\omega ^{2u_{1}-2u_{2}}c_{2} & \omega
^{2u_{2}-2u_{1}}\left( c_{1}-\omega ^{2u_{1}-2u_{2}}c_{2}\right) ^{2} \\ 
0 & 1 & 2\omega ^{2u_{2}-2u_{1}}\left( c_{1}-\omega
^{2u_{1}-2u_{2}}c_{2}\right) \\ 
0 & 0 & \omega ^{2u_{2}-2u_{1}}%
\end{array}%
\right) . 
\]%
If $m=0$, then $\tau_{\xi} \alpha ^{2s}\tau_{\xi} $ is represented by%
\scriptsize
\[
N_{0}\allowbreak =\left( 
\begin{array}{ccc}
\frac{1}{\omega ^{2s}}\left( \xi ^{2}+\omega ^{2s}\right) ^{2} & \frac{\xi }{%
\omega ^{2s}}\left( \omega ^{2s}-1\right) \left( \xi ^{2}+\omega ^{2s}\right)
& \frac{\xi ^{2}}{\omega ^{2s}}\left( \omega ^{2s}-1\right) ^{2} \\ 
2\frac{\xi }{\omega ^{2s}}\left( \omega ^{2s}-1\right) \left( \xi
^{2}+\omega ^{2s}\right) & \frac{1}{\omega ^{2s}}\left( 2\xi ^{2}+\omega
^{2s}-2\xi ^{2}\omega ^{2s}+2\xi ^{2}\omega ^{4s}+\xi ^{4}\omega ^{2s}\right)
& 2\frac{\xi }{\omega ^{2s}}\left( \omega ^{2s}-1\right) \left( \xi
^{2}\omega ^{2s}+1\right) \\ 
\frac{\xi ^{2}}{\omega ^{2s}}\left( \omega ^{2s}-1\right) ^{2} & \frac{\xi }{%
\omega ^{2s}}\left( \omega ^{2s}-1\right) \left( \xi ^{2}\omega
^{2s}+1\right) & \frac{1}{\omega ^{2s}}\left( \xi ^{2}\omega ^{2s}+1\right)
^{2}%
\end{array}%
\right) . 
\]%
\normalsize
Then $\gamma _{c_{1}}\alpha ^{2(u_{1}-u_{2})}\gamma _{c_{2}}^{-1}=\tau_{\xi} \beta
^{m}\alpha ^{2s}\tau_{\xi} $ implies $N_{0}=\theta M$ for some $\theta \in
GF(q)^{\ast }$, and hence $\omega ^{2s}=1$. Thus, $\gamma _{c_{1}}\alpha
^{2(u_{1}-u_{2})}\gamma _{c_{2}}^{-1}=1$ and hence $u_{1}=u_{2}$ and $%
c_{1}=c_{2}$.

If $m=1$, then $\tau_{\xi} \beta \alpha ^{2s}\tau_{\xi} $ is represented by 
\scriptsize
\[
N_{1}=\left( 
\begin{array}{ccc}
\frac{\xi ^{2}}{\omega ^{2s}}\left( \omega ^{2s}-1\right) ^{2} & \frac{\xi }{%
\omega ^{2s}}\left( \omega ^{2s}-1\right) \left( \xi ^{2}\omega
^{2s}+1\right) & \frac{1}{\omega ^{2s}}\left( \xi ^{2}\omega ^{2s}+1\right)
^{2} \\ 
-2\frac{\xi }{\omega ^{2s}}\left( \omega ^{2s}-1\right) \left( \xi
^{2}+\omega ^{2s}\right) & -\frac{1}{\omega ^{2s}}\left( 2\xi ^{2}+\omega
^{2s}-2\xi ^{2}\omega ^{2s}+2\xi ^{2}\omega ^{4s}+\xi ^{4}\omega ^{2s}\right)
& -2\frac{\xi }{\omega ^{2s}}\left( \omega ^{2s}-1\right) \left( \xi
^{2}\omega ^{2s}+1\right) \\ 
\frac{1}{\omega ^{2s}}\left( \xi ^{2}+\omega ^{2s}\right) ^{2} & \frac{\xi }{%
\omega ^{2s}}\left( \omega ^{2s}-1\right) \left( \xi ^{2}+\omega ^{2s}\right)
& \frac{\xi ^{2}}{\omega ^{2s}}\left( \omega ^{2s}-1\right) ^{2}%
\end{array}%
\right) . 
\]%
\normalsize
Then $\gamma _{c_{1}}\alpha ^{2(u_{1}-u_{2})}\gamma _{c_{2}}^{-1}=\tau_{\xi} \beta
\alpha ^{2s}\tau_{\xi} $ implies $N_{1}=\theta M$ for some $\theta \in GF(q)^{\ast
}$, and hence $\omega ^{2s}=-\xi ^{2}$, from which we derive 
\[
N_{1}=\left( 
\begin{array}{ccc}
-\left( \xi ^{2}+1\right) ^{2} & -\frac{1}{\xi }\left( \xi ^{2}+1\right)^{2}
\left( \xi ^{2}-1\right) & -\frac{1}{\xi ^{2}}\left( \xi ^{2}+1\right) ^{2}\left( \xi ^{2}-1\right) ^{2}
\\ 
0 & \left( \xi ^{2}+1\right) ^{2} & \frac{2}{\xi }\left( \xi ^{2}+1\right)^{2}
\left( \xi ^{2}-1\right) \\ 
0 & 0 & -\left( \xi ^{2}+1\right) ^{2}%
\end{array}%
\right) . 
\]%
Then $\theta =\left( \xi ^{2}+1\right) ^{2}$ and $\omega ^{2u_{1}-2u_{2}}=-1$ since $\omega ^{2u_{1}-2u_{2}}=1$ implies $c_{1}=c_{2}$, hence $q \equiv 1 \pmod{4}$ and $c_{2}=\xi ^{-1}-\xi -c_{1}$. Thus, $T_{B}\tau_{\xi} \gamma
_{c_{1}}\alpha ^{2u_{1}}=T_{B}\tau_{\xi} \gamma _{c_{2}}\alpha ^{2u_{2}}$ with $%
c_{2}=\xi ^{-1}-\xi -c_{1}$ and $u_{2}=u_{1}+\frac{q-1}{4}$, which is (iii).

Finally, assume that $q \equiv 1 \pmod{4}$, $T_{B}\tau _{\xi _{1}}\gamma _{c_{1}}\alpha
^{2u_{1}}=T_{B}\tau _{\xi _{2}}\gamma _{c_{2}}\alpha ^{2u_{2}}$ with $\xi_{1} \neq \xi_{2}$. Then $\gamma _{c_{1}}\alpha
^{2(u_{1}-u_{2})}\gamma _{c_{2}}^{-1}=\tau _{\xi _{1}}\beta ^{m}\alpha
^{2s}\tau _{\xi _{2}}$, and hence $O^{\tau _{\xi _{1}}\beta ^{m}\alpha
^{2s}\tau _{\xi _{2}}}=O$. On the other hand, easy computations show that
\small
\begin{eqnarray*}
O^{\tau _{\xi _{1}}\alpha ^{2s}\tau _{\xi _{2}}} &=&\left( \frac{1}{\omega
^{2s}}\left( \xi _{2}-\omega ^{2s}\xi _{1}\right) ^{2},-\frac{1}{\omega ^{2s}%
}\left( \xi _{2}-\omega ^{2s}\xi _{1}+\omega ^{2s}\xi _{1}\xi
_{2}^{2}-\omega ^{4s}\xi _{1}^{2}\xi _{2}\right) ,\frac{1}{\omega ^{2s}}%
\left( \omega ^{2s}\xi _{1}\xi _{2}+1\right) ^{2}\right),  \\
O^{\tau _{\xi _{1}}\beta \alpha ^{2s}\tau _{\xi _{2}}} &=&\left( \frac{1}{%
\omega ^{2s}}\left( \xi _{1}\xi _{2}+\omega ^{2s}\right) ^{2},-\frac{1}{%
\omega ^{2s}}\left( \omega ^{2s}\xi _{1}-\omega ^{4s}\xi _{2}+\xi
_{1}^{2}\xi _{2}-\omega ^{2s}\xi _{1}\xi _{2}^{2}\right) ,\frac{1}{\omega
^{2s}}\left( \xi _{1}-\omega ^{2s}\xi _{2}\right) ^{2}\right)\text{.} 
\end{eqnarray*}

\normalsize
Hence, either $\xi _{2}-\omega ^{2s}\xi _{1}=0$ or $\xi _{1}\xi _{2}+\omega
^{2s}=0$. The former implies $\xi _{1}\xi _{2}=\omega ^{2s}\xi _{1}^{2}$,
and hence $\chi (\xi _{1}\xi _{2})=1$, the latter $\xi _{1}\xi _{2}=-\omega
^{2s}$ and again $\chi (\xi _{1}\xi _{2})=1$ since $q\equiv 1\pmod{4}$. $\hfill\square$
\end{proof}

\bigskip

\bigskip

If $q \equiv 1 \pmod{4}$, let $\xi _{1},\xi _{2}\in GF(q)^{\ast }$, $\xi _{1} \neq \xi _{2}$,  such that $\xi _{1}^{2},\xi
_{2}^{2}\neq -1$ and $\chi (\xi _{1}\xi _{2})=-1$. Such $\xi
_{1},\xi _{2}$ do exist for $q>5$. For instance, $\xi _{1}=1$ and $\xi
_{2}=\omega $ fulfill the previous properties.

\bigskip

\bigskip

The group $T$ acts block-transitively on $\mathcal{D}$ and it is clearly that the action of $T$ on $\mathcal{B}$ is equivalent to the action $T$ on the
set of cosets of $T_{B}$ in $T$. Now, if $q \equiv -1 \pmod{4}$, then $F_{\xi}=W\cup \tau
_{\xi}WH$, where and $H$ is the cyclic subgroup of order $\frac{q-1}{2}$ of $T_{B}$, is a
system of distinct representatives of the cosets of $T_{B}$ in $T$ by Lemma \ref{LSR1}.

If $q \equiv 1 \pmod{4}$, then $F_{\xi_{1},\xi_{2}}=W\cup \tau
_{\xi _{1}}WH\cup \tau
_{\xi _{2}}WH$
contains a
system of distinct representatives of the cosets of $T_{B}$ in $T$ again by Lemma %
\ref{LSR1}. Hence, we have the following proposition.

\begin{proposition}\label{BF}
Either $q \equiv -1 \pmod{4}$ and $\mathcal{B}=B^{F_{\xi}}$, or $q \equiv 1 \pmod{4}$ and $\mathcal{B}=B^{F_{\xi_{1},\xi_{2}}}$. 
\end{proposition}

Let $i,j \in \{-1,+1\}$ and $\xi$ be $1$ for $q \equiv -1 \pmod{4}$ or an element in $\{\xi_{1},\xi_{2}\}$ for $q \equiv 1 \pmod{4}$. By using (\ref{rep1}) and (\ref{rep2}), one obtains 
\begin{equation}\label{Oc}
\mathcal{O}_{h,i}^{\gamma _{c}} =\left\{ (h\mu _{1}^{2},\left( ch\mu
_{1}+1\right) \mu _{1},hc^{2}\mu _{1}^{2}+2c\mu _{1}+1):\mu _{1}\in
Q_{i}\right\},
\end{equation}
\begin{equation}\label{Ochpm}
\mathcal{O}_{h^{p^{m}},j}^{\gamma _{c}} =\left\{ (h^{p^{m}}\mu
_{2}^{2},\left( ch^{p^{m}}\mu _{2}+1\right) \mu _{2},h^{p^{m}}c^{2}\mu
_{2}^{2}+2c\mu _{2}+1):\mu _{2}\in Q_{j}\right\},
\end{equation}%
\begin{equation}\label{Oh}
\mathcal{O}_{h,i}^{\tau_{\xi} \gamma _{c}\alpha ^{2u}}:\left\{ 
\begin{array}{l}
X_{0}=\omega ^{2u}\left( h\mu _{1}^{2}+2\xi \mu _{1}+\xi ^{2}\right)  \\ 
X_{1}=h\left( \xi +c\right) \mu _{1}^{2}+\left( \xi ^{2}+2c\xi -1\right) \mu
_{1}+\xi \left( c\xi -1\right)  \\ 
X_{2}=\frac{1}{\omega ^{2u}}\left( h \left( c+\xi \right) ^{2}\mu
_{1}^{2}+2\left( c\xi -1\right) \left( c+\xi \right) \mu _{1}+\left( c\xi
-1\right) ^{2}\right) 
\end{array}%
\right. 
\end{equation}
with $\mu _{1}\in Q_{i}$, and 
\begin{equation}\label{Ohpm}
\mathcal{O}_{h,j}^{\tau_{\xi} \gamma _{c}\alpha ^{2u}}:\left\{ 
\begin{array}{l}
X_{0}=\omega ^{2u}\left( h^{p^{m}}\mu _{2}^{2}+2\xi \mu _{2}+\xi ^{2}\right) 
\\ 
X_{1}=h^{p^{m}}\left( \xi +c\right) \mu _{2}^{2}+\left( \xi ^{2}+2c\xi
-1\right) \mu _{2}+ \xi \left( c\xi -1\right)  \\ 
X_{2}=\frac{1}{\omega ^{2u}}\left( h^{p^{m}} \left( c+\xi \right)
^{2}\mu _{2}^{2}+2\left( c\xi -1\right) \left( c+\xi \right) \mu _{2}+\left(
c\xi -1\right) ^{2}\right) 
\end{array}%
\right. 
\end{equation}%
with $\mu _{2}\in Q_{j}$.

\begin{remark}\label{Nozero}
It worth nothing that, the first and the third coordinate of the points in $\mathcal{O}_{h,i}^{\gamma _{c}}, \mathcal{O}_{h^{p^{m}},j}^{\gamma _{c}}, \mathcal{O}_{h,i}^{\tau_{\xi} \gamma _{c}\alpha ^{2u}}$ or $\mathcal{O}_{h,j}^{\tau_{\xi} \gamma _{c}\alpha ^{2u}}$ are never equal to zero since the lines $X_{0}=0$ and $X_{2}=0$ are tangent to $\mathcal{C}$ and these do not contain $I$-points of $\mathcal{C}$ by \cite[Table 8.1]{Hir}.   
\end{remark}

\subsubsection{Reduction to the case $q \equiv 1 \pmod{4}$}

\begin{lemma}
\label{conicsol}If $q\equiv -1\pmod{4}$, for each $h\in 1+Q_{+1}$ the number of elements $t\in \left\{1,...,q-2\right\} $ such
that $\omega ^{2t}\left( h-1\right) +\omega ^{t}\left( 2h+2\right)
+h-1\in Q_{-1}$, and $\omega ^{t}\neq \pm \frac{\sqrt{h}-1}{\sqrt{h}+1},%
\pm \frac{\sqrt{h}+1}{\sqrt{h}-1}$ when $h \in Q_{+1}$, is at least $\frac{q-1}{2}$.
\end{lemma}

\begin{proof}
Let $h\in 1+Q_{+}$, the number of $\omega ^{t}$ with $t\in
\left\{ 1,...,q-2\right\} $ such that 
\[
\omega ^{2t}\left( h-1\right) +\omega ^{t}\left( 2h+2\right)
+h-1=y_{t}^{2}
\]%
for some $y_{t}\in GF(q)$ is equal to the number of lines of $PG(2,q)$ of
the form $m_{t}:X_{0}-\omega ^{t}X_{2}=0$ with $t\in \left\{
1,...,q-2\right\} $ having either $1$ or $2$ points in common with the
following irreducible conic: 
\[
\mathcal{K}:X_{0}^{2}\left( h-1\right) +X_{0}X_{2}\left( 2h+2\right)
+X_{2}^{2}\left( h-1\right) -X_{1}^{2}=0.
\]%
Now, each line $m_{t}$ contains the point $P_{\infty }=(0,1,0)$, and $%
P_{\infty }\notin \mathcal{K}$. The polar line of $P_{\infty }$ with
respect to $\mathcal{K}$ is $\ell :X_{1}=0$. 

If $h \in Q_{-1}$, then $\ell$ is external to $\mathcal{K}$, and hence $P_{\infty}$ is internal to $\mathcal{K}$. The secant lines to $\mathcal{K}$ through $P_{\infty }$
are $\frac{q+1}{2}$ by \cite[Table 8.2]{Hir}. The lines $m_{0}:X_{0}=0$ and $%
m_{\infty }:X_{2}=0$ intersect $\mathcal{K}$ in $(0,\pm \sqrt{h-1},1)$
in $(1,\pm \sqrt{h-1},0)$, respectively. Moreover, the line $m_{q-1}:X_{0}-X_{2}=0$
is external to $\mathcal{K}$ since $h \in Q_{-1}$. Thus, the number of lines
to $\mathcal{K}$ of the form $m_{t}:X_{0}-\omega ^{t}X_{2}=0$ with $t\in
\left\{ 1,...,q-2\right\} $ intersecting $\mathcal{K}$ is $\frac{q+1}{2}-2=%
\frac{q-3}{2}$, and hence the number of the external ones is $q-2-\frac{q-3}{2}=\frac{q-1}{2}$. Thus, the assertion follows in this case.

If $h \in Q_{+1}$, then $\ell \cap 
\mathcal{K}=\left\{ Q_{1,},Q_{2,}\right\} $, where 
\[
Q_{1}=\left( -\frac{\sqrt{h}+1}{\sqrt{h}-1},0,1\right) \text{ and }%
Q_{2}=\left(- \frac{\sqrt{h}-1}{\sqrt{h}+1},0,1\right) \text{,}
\]%
hence $P_{\infty }$ is external to $\mathcal{K}$, and $a_{1}:$ $X_{0}+\frac{%
\sqrt{h}+1}{\sqrt{h}-1}X_{2}=0$ and $a_{2}:X_{0}+\frac{\sqrt{h}-1%
}{\sqrt{h}+1}X_{2}=0$ are the tangents lines to $\mathcal{K}$ through $%
P_{\infty }$. The secant lines to $\mathcal{K}$ through $P_{\infty }$
are $\frac{q-1}{2}$ by \cite[Table 8.2]{Hir}. The lines $m_{0}:X_{0}=0$ and $%
m_{\infty }:X_{2}=0$ intersect $\mathcal{K}$ in $(0,\pm \sqrt{h-1},1)$
in $(1,\pm \sqrt{h-1},0)$, respectively. Moreover, the line $m_{q-1}:X_{0}-X_{2}=0$
intersects $\mathcal{K}$ in $(1,\pm 2\sqrt{h},1)$. Thus, the number of lines
to $\mathcal{K}$ of the form $m_{t}:X_{0}-\omega ^{t}X_{2}=0$ with $t\in
\left\{ 1,...,q-2\right\} $ intersecting $\mathcal{K}$ is $\frac{q-1}{2}-2-1=%
\frac{q-7}{2}$, and hence the number of the external ones is $q-2-\frac{q-7}{2}=\frac{q+3}{2}$, which is equivalent to the number of
elements $t\in \left\{ 2,...,q-1\right\} $ such that $\omega
^{2t}\left( h-1\right) +\omega ^{t}\theta \left( 2h+2\right) +h-1\in
Q_{-1}$. Finally, excluding the remaining possible values $\frac{\sqrt{h}-1}{\sqrt{h}+1},\frac{\sqrt{h%
}+1}{\sqrt{h}-1}$, the assertion follows. $\hfill\square$
\end{proof}

\begin{proposition}\label{q=1mod4}
If $\mathcal{D}$ is a $2$-design of type II, then $q\equiv 1 \pmod{4}$.
\end{proposition}

\begin{proof}
Assume that $q\equiv -1 \pmod{4}$. Let $A_{1}=(w^{2t_{1}},0,1)$ and $A_{2}=(w^{2t_{2}},0,1)$ such that $\omega
^{2t_{0}}\left( h-1\right) +\omega ^{t_{0}}\left( 2h+2\right) +h-1\in Q_{-1}$
and $\omega ^{t_{0}}\neq 1,\pm \frac{\sqrt{h}-1}{ \sqrt{h}+1},\pm \frac{\sqrt{h}+1}{%
\sqrt{h}-1}$, where $t_{0}=t_{1}-t_{2}$. Such $t_{0}$ does exist by Lemma %
\ref{conicsol} since $q\geq 7$. Then there is $\psi \in F_{1}$ by
Proposition \ref{BF} such that $\left\{ A_{1},A_{2}\right\} \subset B^{\psi }$, where $B= \mathcal{C}_{h}^{\ast}$ with $h \in 1+Q_{+1}$ by Lemma \ref{BOO}(1).
Hence, there are two distinct points $P_{1}=(h\mu _{1}^{2},\mu _{1},1)$ and $%
P_{2}=(h\mu _{2}^{2},\mu _{2},1)$ in $B$ such that $P_{1}^{\psi }=A_{1}$ and 
$P_{2}^{\psi }=A_{2}$.

If $\psi \in W$, then $P_{i}^{\psi }=\left( \frac{%
h\mu _{i}^{2}}{hc^{2}\mu _{i}^{2}+2c\mu _{i}+1},\frac{\mu _{i}\left( ch\mu
_{i}+1\right) }{hc^{2}\mu _{i}^{2}+2c\mu _{i}+1},1\right) $ since $hc^{2}\mu
_{i}^{2}+2c\mu _{i}+1\neq 0$ (see Remark \ref{Nozero}), and hence $P_{i}^{\psi }=A_{i}$ for $i=1,2$ if
and only if $c=1/h\mu _{1}=1/h\mu _{2}$. So $\mu _{1}=\mu _{2}$, and we reach a contradiction. Then $\psi =\tau \gamma _{c_{0}}\alpha ^{2u_{0}}$ for some $%
c_{0}\in GF(q)$ and $1\leq u_{0}\leq \frac{q-1}{2}$. This is equivalent to
say that $(\mu _{1},c_{0},u_{0},t_{1})$ and $(\mu _{2},c_{0},u_{0},t_{2})$
are two distinct solutions of the following system obtained by using (\ref{Oh}) and (\ref{Ohpm}) (see also Remark \ref{Nozero}):

\begin{equation}
\left\{ 
\begin{array}{r} 
\left( h\mu ^{2}+2\mu +1\right) c-\left( h\mu ^{2}-1\right)  =0,
\\
\omega ^{2u}\frac{h\mu ^{2}+2\mu +1}{h\allowbreak \left( c+1\right) ^{2}\mu
^{2}+2\left( c^{2}-1\right) \mu _{1}+\left( c-1\right) ^{2}} =\omega ^{2t}.
\end{array} \right.
\label{sistema}
\end{equation}
If $2h\mu _{i}^{2}+\mu _{i}+1=0$ with $i$ either $1$ or $2$, then $h\mu
_{i}^{2}-1=0$. Thus $\mu _{i}=-1$, and hence $h=1$, a contradiction since $h \in 1+Q_{+1}$.
Therefore, $2\mu _{i}+h\mu _{i}^{2}+1\neq 0$ for each $i=1,2$. Then
substituting $(\mu _{1},c_{0},u_{0})$ and $(\mu _{2},c_{0},u_{0})$ in the
first equation of (\ref{sistema}) one obtains  
\[
c_{0}=-\frac{h\mu _{1}^{2}-1}{h\mu _{1}^{2}+2\mu _{1}+1}=-\frac{h\mu
_{2}^{2}-1}{h\mu _{2}^{2}+2\mu _{2}+1}\text{,}
\]%
from which we derive $h(\mu _{1}+1)\mu _{2}=-\left( h\mu _{1}+1\right) $.
The previous argument leads to a contradiction if $\mu _{1}=-1$, hence $\mu
_{1}\neq -1$. This implies 
\begin{equation}
c_{0}=-\frac{h\mu _{1}^{2}-1}{h\mu _{1}^{2}+2\mu _{1}+1}\text{ and }\mu
_{2}=-\frac{h\mu _{1}+1}{h(\mu _{1}+1)}  \label{restric}
\end{equation}%
since $h\neq 0$. Note that $\mu _{2}\neq 0$, hence $h\mu _{1}+1\neq 0$. Now,
substituting $(\mu _{1},c_{0},u_{0},t_{1})$ and $(\mu _{2},c_{0},u_{0},t_{2})
$, with $c_{0}$ and $\mu _{2}$ as in (\ref{restric}), in the second equation
of (\ref{sistema}), we get%
\[
\omega ^{2u}=\frac{4\mu _{1}^{2}\left( h-1\right) }{\left( h\mu
_{1}^{2}+2\mu _{1}+1\right) ^{2}}\omega ^{2t_{1}}\text{ and }\omega ^{2u}=\frac{%
 4\left( \mu _{1}+1\right) ^{2}\left( h\mu _{1}+1\right) ^{2}}{%
\left( h-1\right) \left( h\mu _{1}^{2}+2\mu _{1}+1\right) ^{2}}\omega
^{2t_{2}} \text{.}
\]%
If $\omega ^{u_{1}}=\frac{2\mu _{1}\sqrt{\left( h-1\right) }}{ h\mu _{1}^{2}+2\mu
_{1}+1 }\omega ^{t_{1}}$ and $\omega ^{u_{1}}=-\frac{%
2\left( \mu _{1}+1\right) \left( h\mu _{1}+1\right) }{\sqrt{%
\left( h-1\right) }\left( h\mu _{1}^{2}+2\mu _{1}+1\right) }\omega ^{t_{2}}$,
then 
\[
\frac{2\mu _{1}\sqrt{\left( h-1\right) }}{ h\mu
_{1}^{2}+2\mu _{1}+1 }\omega ^{t_{1}}=-\frac{2\left( \mu
_{1}+1\right) \left( h\mu _{1}+1\right) }{\sqrt{\left( h-1\right) }\left(
h\mu _{1}^{2}+2\mu _{1}+1\right) }\omega ^{t_{2}} \text{,}
\]%
from which we deduce 
\[
\frac{\mu _{1}\left( h-1\right) }{\left( \mu _{1}+1\right) \left( h\mu
_{1}+1\right) }=-\omega ^{t_{0}}\text{.}
\]%
This leads
\[
h\omega ^{t_{0}}\mu _{1}^{2}+\left( \omega ^{t_{0}}\left( h+1\right)
+h-1\right) \mu _{1}+ \omega ^{t_{0}} =0 \text{,}
\]%
which has solutions if and only if 
\[
\left( \omega ^{t_{0}}\left( h+1\right) +h-1\right) ^{2}-4\left( h\omega
^{t_{0}}\right) \omega ^{t_{0}}\in Q_{+1}\cup \left\{ 0\right\} \text{.}
\]%
Easy computations yield 
\[
\left( h-1\right) \left( \allowbreak \omega ^{2t_{0}}\left( h-1\right)
+\omega ^{t_{0}}2\left( h+1\right) +h-1\right) \in Q_{+1}\cup \left\{
0\right\} \text{,}
\]%
and hence%
\[
\omega ^{2t_{0}}\left( h-1\right) +\omega ^{t_{0}}2\left( h+1\right) +h-1\in
Q_{+1}\cup \left\{ 0\right\} 
\]%
since $h\in 1+Q_{+}$, which is not the case by our assumption on $\omega
^{t_{0}}$. \ The same conclusion holds for $\omega ^{u_{1}}=-\frac{2\mu _{1}\sqrt{%
\left( h-1\right) }}{ h\mu _{1}^{2}+2\mu _{1}+1 }\omega
^{t_{1}}$ and $\omega ^{u_{1}}=\frac{ 2\left( \mu _{1}+1\right)
\left( h\mu _{1}+1\right) }{\sqrt{\left( h-1\right) }\left( h\mu
_{1}^{2}+2\mu _{1}+1\right) }\omega ^{t_{2}}$.

Thus $\omega ^{u_{1}}=\frac{2\mu _{1}\sqrt{\left( h-1\right) }}{h\mu _{1}^{2}+ 2\mu
_{1}+1 }\omega ^{t_{1}}$, $\omega ^{u_{1}}=\frac{%
 2\left( \mu _{1}+1\right) \left( h\mu _{1}+1\right) }{\sqrt{%
\left( h-1\right) }\left( h\mu _{1}^{2}+2\mu _{1}+1\right) }\omega ^{t_{2}}$
and $\omega ^{u_{2}}=-\omega ^{u_{1}}$. Now, arguing as above, one obtains%
\begin{equation}
h\omega ^{t_{0}}\mu _{1}^{2}+\left( \omega ^{t_{0}}\left( h+1\right)
+h-1\right) \mu _{1}-\omega ^{t_{0}} =0  \label{equat}
\end{equation}%
with 
\[
\omega ^{2t_{0}}\left( h-1\right) -\omega ^{t_{0}}2\left( h+1\right) +h-1\in
Q_{+1}\cup \left\{ 0\right\} \text{.}
\]%
If $\omega ^{2t_{0}}\left( h-1\right) -\omega ^{t_{0}}2\left( h+1\right)
+h-1=0$, then $\omega ^{t_{0}}=\frac{\sqrt{h}-1}{\sqrt{h}+1}$ or $\frac{%
\sqrt{h}+1}{\sqrt{h}-1}$ which is still contrary to our assumption. Therefore 
\[
\omega ^{2t_{0}}\left( h-1\right) -\omega ^{t_{0}}2\left( h+1\right) +h-1\in
Q_{+1} \text{,}
\]%
and hence there are two distinct solutions $\mu _{1},\mu _{1}^{\prime }$ of (%
\ref{equat}). It results that
\begin{eqnarray*}
c(x) &=&-\frac{hx^{2}-1}{hx^{2}+2x+1}\text{, }\mu _{2}(x)=-\frac{hx+1}{h(x+1)%
}\text{ and }\omega ^{u_{1}}(x)=\frac{2x\sqrt{\left( h-1\right) }}{
hx^{2}+2x+1 }\omega ^{t_{1}}, \\
u_{2} &=&u_{1}+\frac{q-1}{2}
\end{eqnarray*}%
with $x=\mu _{1},\mu _{1}^{\prime }$ are solutions of (\ref{sistema}). Hence,
$\left\{ P_{1}(x),P_{2}(x)\right\} ^{\tau \gamma _{c(x)}\alpha
^{2u_{i}(x)}}=\left\{ A_{1},A_{2}\right\} $ for $i=1,2$, where $P_{1}(x)=(hx^{2},hx,1)$
and $P_{2}(x)=(h\mu _{2}(x)^{2},h\mu _{2}(x),1)$ and $x=\mu _{1},\mu
_{1}^{\prime }$. Moreover, the set $\left\{ P_{1}(\mu _{1}),
P_{2}(\mu _{1})\right\}\neq \left\{ P_{1}(\mu _{1}^{\prime }),P_{2}(\mu _{1}^{\prime })\right\}$
since $\mu _{1}\neq \mu _{1}^{\prime }$. Note that $\tau _{1}\gamma
_{c(x)}\alpha ^{2u_{i}(x)}$ with $x=\mu _{1},\mu _{1}^{\prime }$ and $i=1,2$
are four distinct elements of $F_{1}$ by Lemma \ref{LSR1}(iii). Then $B^{\tau _{1}\gamma
_{c(x)}\alpha ^{2u_{i}(x)}}$ with $x=\mu _{1},\mu _{1}^{\prime }$ and $i=1,2$ are four distinct blocks of $\mathcal{D}$ by
Proposition \ref{BF}, and all of them are such that $\left\{ A_{1},A_{2}\right\} \subseteq
B^{\tau \gamma _{c(x)}\alpha ^{2u_{i}(x)}}$. This violates $\lambda =2$.
Thus $q \equiv -1 \pmod{4}$ is ruled out, and hence the assertion follows since $q$ is odd.    
\end{proof}

\subsubsection{Reduction to the case $q=5$}

In this final section, we complete the proof of Theorem \ref{main} by handling the remaining case $q \equiv 1 \pmod{4}$.
\begin{lemma}\label{Syl}
The number of unordered pairs of distinct points of $\mathcal{O}_{\infty }$ contained in a block of $\mathcal{D}$ of the form $B^{\gamma_{c}}$ with $c\in GF(q)^{\ast }$ is $\frac{q-1}{4}$ for $B=\mathcal{O}_{h,i}\cup \mathcal{O}_{h^{p^{m}},i}$, $i \in \{-1,+1\}$, and $0$ otherwise.
\end{lemma}

\begin{proof}
Let $B$ be any block of $\mathcal{D}$, then either $B=\mathcal{C}_{h}^{\ast}$ for some $h \in 1+Q_{-1}$, or $B=\mathcal{O}_{h,i}\cup \mathcal{O}_{h^{p^{m}},j}$ for some $h\in (1+Q_{-1}) \cap Q_{+1}$, $h^{p^{m}}\neq h$ and $h^{p^{2m}}=h$, and $i,j\in \{-1,+1\} $ by Lemma \ref{BOO}(2). In the former case, for each $c \in GF(q)^{\ast}$, it is easy to see that $B^{\gamma_{c}} \cap \mathcal{O}_{\infty }=\{P_{c}\}$ where $P_{c}=\left(h\left( 1 /ch\right) ^{2},1 /ch,1\right)$ and $P_{c}^{\gamma _{c}}=(1/c^{2}(h-1),0,1)$. Thus, the number of unordered pairs of distinct points of $\mathcal{O}_{\infty }$ contained in a block of $\mathcal{D}$ of the form $B^{\gamma_{c}}$ with $c\in GF(q)^{\ast }$ is $0$ in this case.

Assume that $B=\mathcal{O}_{h,i}\cup \mathcal{O}_{h^{p^{m}},j}$ for some $h\in (h \in 1+Q_{-1}) \cap Q_{+1}$, $h^{p^{m}}\neq h$ and $h^{p^{2m}}=h$, and $i,j\in \{-1,+1\} $. Let $c \in GF(q)^{\ast}$, then $P_{c}$ as above and  $R_{c}=\left( h^{p^{m}}\left(1 /ch^{p^{m}}\right)^{2},1/ch^{p^{m}},1\right)$ are the unique $I$-points of $\mathcal{C}$  which are
mapped by $\gamma _{c}$ onto points of $\ell : X_{1}=0$ (here, $R_{c}^{\gamma _{c}}=(1/c^{2}(h^{p^{m}}-1),0,1)$). Further, $P_{c} \neq R_{c}$ since $h^{p^{m}}\neq h$, hence $P_{c}^{\gamma_{c}} \neq R_{c}^{\gamma_{c}}$. 

Since $h\in Q_{+}$ and $B=\mathcal{O}_{h,i}\cup \mathcal{O}_{h^{p^{m}},j}$, it follows that both $P_{c}$ and $R_{c}$ lie in $B$ if and only if $i=j$ and $c\in Q_{i}$. Thus, $\left\vert B^{\gamma_{c}}\cap \mathcal{O}_{\infty }\right\vert=2$ if and only if $i=j$ and $c \in Q_{i}$. When this occurs, $B^{\gamma_{c}}\cap \mathcal{O}_{\infty }=B^{\gamma_{-c}}\cap \mathcal{O}_{\infty }$ since $P_{-c}^{\gamma _{-c}}=P_{c}^{\gamma_{c}}$ and 
$R_{-c}^{\gamma _{-c}}=R_{c}^{\gamma_{c}}$. Therefore, the number of unordered pairs of distinct points of $\mathcal{O}_{\infty }$ contained in a block of $\mathcal{D}$ of the form $B^{\gamma_{c}}$ with $c\in Q_{i}$ is $\frac{q-1}{4}$ since $\left \vert Q_{i} \right\vert =\frac{q-1}{2}$ and for each $c \in Q_{i}$ one has $B^{\gamma_{c}}\cap \mathcal{O}_{\infty }=B^{\gamma_{-c}}\cap \mathcal{O}_{\infty }$.  $\hfill\square$
\end{proof}

\bigskip

{\bf Proof of Theorem \ref{main}.}\, The aim of this proof is to demonstrate that there are no $2$-$\left( \frac{q(q-1)}{2},q-1,2\right) $ designs $\mathcal{D}$
admitting $PSL(2,q)\trianglelefteq G\leq P\Gamma L(2,q)$ as a flag-transitive automorphims group when $q>5$. We will prove this by showing that there are pairs of points such that if there is block containing them, then the number distinct blocks containing them is at least four.

The $I$-points of $\mathcal{C}$ contained in $\ell $ are exactly those of $%
\mathcal{O}_{\infty }$, which are $\frac{q-1}{2}$ points. Hence, the number of
unordered pairs of distinct points of $\mathcal{O}_{\infty }$ is $\frac{%
(q-1)(q-3)}{8}$. Let $\{A_{1},A_{2}\}$ be any such pair, then there is an element of $\mathcal{B}$ containing it since $\mathcal{D}$ is a $2$-design, and hence there is $\psi \in F_{\xi_{1},\xi_{2}}$ such that $A_{1},A_{2} \in B^{\psi}$ by Proposition \ref{BF}, where either $B=\mathcal{O}_{h,i}\cup \mathcal{O}_{h^{\prime},j}$ with $h^{\prime}=h$ and  $i,j \in \{-1,+1\}$, $i\neq j$, and hence $B=\mathcal{C}_{h}^{\ast}$, or $h^{\prime}=h^{p^{m}} \neq h$, $h^{p^{2m}}=h$ for some $1 \leq m \leq f/2$, and $i,j \in \{-1,+1\}$ by Lemma \ref{BOO}(2). 

Note that, $\psi \neq 1$ since $B \cap \ell=  \emptyset$ and $ \mathcal{O}_{\infty } \subset \ell$, where $\ell:X_{1}=0$. Now, $\psi \in W$ if and only if $B=\mathcal{O}_{h,i}\cup \mathcal{O}_{h^{p^{m}},i}$ by Lemma \ref{Syl}. Further, the set $Y$ consisting of the unordered pairs of distinct points of $%
\mathcal{O}_{\infty }$ not contained in a block of the form $B^{\gamma _{c}}$ with $c \in GF(q)^{\ast}$, is $\frac{(q-1)(q-3)}{8}-\frac{q-1}{4}$ or  $\frac{%
(q-1)(q-3)}{8}$ according as $B$ is equal or not to $\mathcal{O}_{h,i}\cup \mathcal{O}_{h^{p^{m}},i}$, respectively. Thus $\left\vert Y \right \vert >1$ since $q\equiv 1\pmod{4}$ and $q>5$, and hence each pair $\{A_{1},A_{2} \} \in X$ is contained in $B^{\psi}$ for some $\psi=\tau_{\xi} \gamma _{c_{0}}\alpha ^{2u_{0}}$ with $\xi \in \{ \xi_{1}, \xi_{2}\}$, $c_{0}\in GF(q)$ and $1\leq u_0\leq \frac{q-1}{2}$ by Proposition \ref{BF}.

Now,  $A_{1}=(\omega
^{2t_{1}+1},0,1)$ and $A_{2}=(\omega ^{2t_{2}+1},0,1)$ for some fixed $%
t_{1},t_{2} \in \left\{ 1,...,\frac{q-1}{2}\right\} $ with $t_{1}\neq t_{2}$ by Lemma \ref{orbit}(ii.b) since $A_{1}$ and $%
A_{2}$ are distinct $I$-points of $\mathcal{C}$ lying in $\mathcal{O}_{\infty }$.
Then there are $%
P_{1},P_{2}\in B$ such that $P_{1}^{\psi }=A_{1}$ and $P_{2}^{\psi }=A_{2}$.
Hence, there are $\mu _{1},\mu _{2}\in GF(q)^{\ast }$ such that either $%
P_{1}=\left( h\mu _{1}^{2},\mu _{1},1\right) $ and $P_{2}=\left( h^{\prime}\mu
_{2}^{2},\mu _{2},1\right) $, or $P_{1}=\left( h\mu _{1}^{2},\mu
_{1},1\right) $ and $P_{2}=\left( h\mu _{2}^{2},\mu _{2},1\right) $, or $%
P_{1}=\left( h^{\prime}\mu _{1}^{2},\mu _{1},1\right) $ and $P_{2}=\left(
h^{\prime}\mu _{2}^{2},\mu _{2},1\right) $ by Lemma \ref{BOO}.

Assume that $P_{1}=\left( h\mu _{1}^{2},\mu _{1},1\right) $ and $%
P_{2}=\left( h^{\prime}\mu _{2}^{2},\mu _{2},1\right) $. Then $P_{1}^{\psi
}=A_{1}$ and $P_{2}^{\psi}=A_{2}$ is equivalent to say that for the
fixed $t_{1},t_{2} \in \left\{ 1,...,\frac{q-1}{2}\right\} $ with $t_{1}\neq t_{2}$ there is
a solution $(c_{0},u_{0},\mu _{1},\mu _{2})$ of the following system of equations determined by using (\ref{Oh}) and (\ref{Ohpm}) (see also Remark \ref{Nozero}): 
\begin{equation}
\left\{ 
\begin{array}{c}
h\left( \xi +c\right) \mu _{1}^{2}+\left( \xi ^{2}+2c\xi -1\right) \mu
_{1}+\xi \left( c\xi -1\right)=0,
\\ 
\omega ^{4u} \frac{ h\mu _{1}^{2}+2\xi \mu _{1}+\xi ^{2}}{ h \left( c+\xi \right) ^{2}\mu
_{1}^{2}+2\left( c\xi -1\right) \left( c+\xi \right) \mu _{1}+\left( c\xi
-1\right) ^{2}}=\omega ^{2t_{1}+1}, \\ 
h^{\prime}\left( \xi +c\right) \mu _{2}^{2}+\left( \xi ^{2}+2c\xi
-1\right) \mu _{2}+ \xi \left( c\xi -1\right) =0,\\ 

\omega ^{4u}\frac{ h^{\prime}\mu _{2}^{2}+2\xi \mu _{2}+\xi ^{2} }{h^{\prime} \left( c+\xi \right)
^{2}\mu _{2}^{2}+2\left( c\xi -1\right) \left( c+\xi \right) \mu _{2}+\left(
c\xi -1\right) ^{2}}=\omega
^{2t_{2}+1}.%
\end{array}%
\right.   \label{System}
\end{equation}

Since $q\equiv 1\pmod{4}$, it follows that there are exactly four
distinct solutions of the equations $\omega ^{4u}=\omega ^{4u_{0}}$, say $%
\omega ^{u_{e}}$ with $e=0,1,2,3$. In particular, $u_{0}$, $%
u_{1},u_{2},u_{3} $ are pairwise distinct. Hence, for the fixed $t_{1},t_{2} \in \left\{ 1,...,\frac{q-1}{2}\right\} $ with $t_{1}\neq t_{2}$
the quadruples $(c_{0},u_{i},\mu _{1},\mu _{2})$ with $i=1,2,3$ are also
solutions of (\ref{System}). Therefore, $P_{1}^{\psi _{e}}=A_{1}$ and $%
P_{2}^{\psi _{e}}=A_{2}$ with $\psi _{e}=\tau_{\xi} \gamma _{c_{0}}\alpha ^{2u_{e}}
$ and $\psi_{0}=\psi$, and hence $A_{1},A_{2}\in B^{\psi _{e}}$ for $e=0,1,2,3$. Further, $\psi
_{e}$ with $e=0,1,2,3$ are four distinct elements of $F_{\xi_{1},\xi_{2}}$.

If $c_{0} \neq \frac{1}{2}(\xi^{-1}-\xi)$, then $B^{\psi _{e}}
$ with $e=0,1,2,3$ are four distinct blocks of $\mathcal{D}$ by Lemma \ref{LSR1}(3), and each of these contains $\{A_{1},A_{2}\}$. This violates $\lambda =2$. Therefore, each element in $Y$ in contained in exactly $2$ blocks of the form $B^{\tau_{\xi} \gamma _{c_{0}}\alpha ^{2u_{0}}}$ with $c_{0} = \frac{1}{2}(\xi^{-1}-\xi)$ since $\mathcal{D}$ is a $2$-design with $\lambda=2$.

The number of elements in $F_{\xi_{1},\xi_{2}}$ with $c_{0}= \frac{1}{2}(\xi^{-1}-\xi)$ and $\xi \in \{\xi_{1},\xi_{2}\}$ is $q-1$ by Lemma \ref{LSR1}(3)(4) since $\chi(\xi_{1}\xi_{2})=-1$, and these are such that $B^{\tau_{\xi} \gamma _{c_{0}}\alpha ^{2u_{0}}}=B^{\tau_{\xi} \gamma _{c_{0}}\alpha ^{2u_{0}+\frac{q-1}{2}}}$. Hence, the number of distinct blocks of $\mathcal{D}$ of the form $B^{\tau_{\xi} \gamma _{c_{0}}\alpha ^{2u_{0}}}$ with $c_{0}= \frac{1}{2}(\xi^{-1}-\xi)$ and $\xi \in \{\xi_{1},\xi_{2}\}$ is $\frac{q-1}{2}$ by Lemma \ref{LSR1}(3). Then $\left\vert Y\right\vert \leq \frac{q-1}{4}$ since each element in $Y$ is contained in exactly $2$ blocks of the form $B^{\tau_{\xi} \gamma _{c_{0}}\alpha ^{2u_{0}}}$ with $c_{0}= \frac{1}{2}(\xi^{-1}-\xi)$ and $\xi \in \{\xi_{1},\xi_{2}\}$. Then $q \leq 7$ since $\left\vert Y \right\vert $ is either $\frac{%
(q-1)(q-3)}{8}$ or $\frac{(q-1)(q-3)}{8}-\frac{q-1}{4}$, but this contradicts $q>5$ and $q \equiv 1 \pmod{4}$. The cases $P_{1}=\left( h\mu _{1}^{2},\mu _{1},1\right) $ and $P_{2}=\left( h\mu _{2}^{2},\mu _{2},1\right) $, or $%
P_{1}=\left( h^{\prime}\mu _{1}^{2},\mu _{1},1\right) $ and $P_{2}=\left(
h^{\prime}\mu _{2}^{2},\mu _{2},1\right) $ are ruled out similarly. This completes the proof of Theorem \ref{main}.     $\hfill\square$

\bigskip

\noindent
{\bf Proof of Theorem \ref{TH1}.}
 It follows immediately from Proposition \ref{D2(q+1)even} and Theorems \ref{soon} and \ref{main}. $\hfill\square$

\bigskip

\noindent{\bf Declaration of Competing Interest}

The authors declare that they have no known competing financial interests or personal relationships that could have appeared to influence the work reported in this paper.

\noindent{\bf Data availability}

No data was used for the research  described in the article.

\end{document}